\documentclass[a4paper]{article}
\usepackage{amsfonts}
\usepackage{latexsym,bm}
\usepackage{amssymb,amsmath,amsthm}
\usepackage{mathrsfs}
\usepackage[english,german]{babel}
\usepackage[colorlinks=true]{hyperref}
\usepackage{color}

\newtheorem{theorem}{Theorem}[section]
\newtheorem{lemma}[theorem]{Lemma}
\newtheorem{remark}[theorem]{Remark}
\newtheorem{proposition}[theorem]{Propositon}

\newtheorem{corollary}[theorem]{Corollary}
\numberwithin{equation}{section}
\hypersetup{linkcolor=blue,urlcolor=red,citecolor=red}

\makeatletter
\def\author#1{\gdef\autrun{\def\and{\unskip, }#1}\gdef\@author{#1}}

\makeatother

\def\code#1{\par\medskip
\noindent\textbf{Mathematics Subject Classification (2010).} #1}

\begin{document}
\selectlanguage{english}

\title{Asymptotic observability identity for the heat equation in $\mathbb{R}^d$}

\author{Gengsheng Wang\thanks{Center for Applied Mathematics, Tianjin University,
Tianjin, 300072, China. \emph{Email: wanggs62@yeah.net}}\quad
Ming Wang\thanks{School of Mathematics and Physics, China University of Geosciences, Wuhan, 430074,  China. \emph{Email: mwangcug@outlook.com}}\quad
Yubiao Zhang\thanks{Center for Applied Mathematics, Tianjin University,
Tianjin, 300072, China. \emph{Email: yubiao\b{ }zhang@yeah.net}}
}

\date{}

\maketitle
\begin{abstract}
We build up   an asymptotic observability identity for
 the heat  equation in the whole space.
 It says that one can approximately recover a solution, through observing it over some countable lattice points in the space and at one time.
 This asymptotic  identity
     is a natural extension of   the well-known Shannon-Whittaker sampling theorem \cite{Shannon,Whittaker}.
 According to  it,   we obtain  a kind of feedback null approximate controllability for  impulsively controlled heat equations.
 We also obtain a weak asymptotic observability identity with finitely many observation
lattice points. This  identity holds only for some solutions to the heat equation.

\end{abstract}

\begin{keywords}
Heat equation, sampling theorem, observability, controllability

\end{keywords}

\code{
35Q93, 93B07
}

\section{Introduction}

The well-known Shannon-Whittaker sampling theorem \cite{Shannon,Whittaker} (see also \cite{H,LM,P,Y}) says that
any function $f$ in the Wiener class (i.e., $f\in L^2(\mathbb{R})$
and  its Fourier transform\footnote{ The Fourier transform of a function $f \in L^1(\mathbb R^d;\mathbb C)$ is given by
$$
\mathcal{F}(f)(\xi)=\hat{f}(\xi):= (2\pi)^{-\frac{d}{2}} \int_{\mathbb{R}^d} e^{-ix\cdot\xi}f(x)\,\mathrm dx,~\xi\in\mathbb R^d.
$$
This transform is extended to all  tempered distributions in  $\mathcal S^\prime(\mathbb R^d;\mathbb C)$ in the usual way.} $\hat{f}$  has a compact support in $\mathbb{R}$) can be completely determined by its values on some  lattice points in $\mathbb{R}$.
Given $N>0$ and  $n=(n_1,\dots,n_d)\in\mathbb Z^d$, we define
\begin{align}\label{equ-basis}
f_{N,n}(x):=\prod_{j=1}^d\frac{\sin\pi(Nx_j-n_j)}{\pi(Nx_j-n_j)}, \quad x=
(x_1,\dots,x_d)\in \mathbb{R}^d,
\end{align}
and
\begin{eqnarray}\label{wang1}
\mathscr{P}_{N} :=
\left\{ f\in L^2(\mathbb{R}^d) ~:~ \mbox{supp}\, \hat{f} \subset Q_{\pi N}(0) \right\},
\end{eqnarray}
where $Q_{\pi N}(0)$ denotes the closed cube in $\mathbb{R}^d$, centered at the origin
and of side length $2\pi N$.
(Some properties on $\{f_{N,n}\}_{n\in\mathbb{Z}^d}$ are given in Lemma \ref{lem-orth}
in Appendix of this paper.)
A high-dimensional version of the Shannon-Whittaker sampling theorem is  as follows:
\begin{theorem}\label{lem-shannon}
For each  $N>0$, any function $f\in\mathscr{P}_N$ satisfies that
\begin{eqnarray}\label{equ-shannon}
f(\cdot) = \sum_{n\in \mathbb{Z}^d} f({n}/{N})
f_{N,n}(\cdot)\;\;\mbox{in}\;\;L^2(\mathbb{R}^d)
\;\;\mbox{and}\;\;
\int_{\mathbb{R}^d}|f(x)|^2\,\mathrm dx = N^{-d} \sum_{n\in \mathbb{Z}^d} \big|f({n}/{N})\big|^2.
\end{eqnarray}
\end{theorem}
\noindent (Theorem \ref{lem-shannon}  is almost the same as Theorem 6.6.9 in
\cite{G}. For the sake of the completeness of the paper, we will give its proof in
 Appendix.)
In this paper, we will  extend the first identity in (\ref{equ-shannon}) to
an asymptotic observability identity for
functions $u(T,\cdot)$, where $T>0$ and $u$ solves
 the heat  equation:
\begin{align}\label{equ-heat}
\partial_tu(t,x)=\triangle u(t,x), \quad (t,x)\in  \mathbb{R}^+\times \mathbb{R}^d; \quad \quad u(0,\cdot)\in L^2(\mathbb{R}^d).
\end{align}

Throughout this paper, we use $\|\cdot\|$ and $\langle \cdot, \cdot\rangle$ to denote the  norm and the inner product of $L^2(\mathbb{R}^d)$;
  we use   $Q_r(x)$ and
$B_r(x)$, with $x\in  \mathbb{R}^d$ and  $r>0$, to denote respectively
the closed cube in $\mathbb{R}^d$ (centered at $x$ and of side length $2r$) and the closed ball in $\mathbb{R}^d$ (centered at $x$ and of radius of $r$); we let $\mathbb{N}^+:=\{1,2,\dots\}$; we use $\hat f$ or $\mathcal{F}(f)$ to denote the Fourier transform of a function $f$; we use $\{e^{t\triangle}\}_{t\geq 0}$ to denote the semigroup on $L^2(\mathbb{R}^d)$,  generated by the operator $\triangle$ (with its domain $D(\triangle)=H^2(\mathbb{R}^d)$).

 The  main theorem of this paper is as follows:

\begin{theorem}\label{thm-decomp}
With the notations in  (\ref{equ-basis}) and (\ref{wang1}),
there is a positive constant $C=C(d)$, depending only on $d$, so that
any solution  $u$ to \eqref{equ-heat} has  the following properties:

 \noindent (i) Given $T>0$ and $N>0$, there is   $R(\cdot)\in L^2(\mathbb R^d)$,
  with
 \begin{eqnarray}\label{wang6}
 \|R(\cdot)\|\leq C\left(1+\left(TN^2\right)^{-\frac{d}{4}} \right)e^{-TN^2}\|u(0,\cdot)\|,
 \end{eqnarray}
       so that
 \begin{align}\label{equ-decomp-725}
 u(T,\cdot) = \sum_{n\in \mathbb{Z}^d} u\big(T,{n}/{N}\big)f_{N,n}(\cdot)+R(\cdot) \;\;\mbox{in}\;\; L^2(\mathbb R^d).
 \end{align}

\noindent  (ii) If $T>0$ and $N>0$, then
 \begin{eqnarray}\label{equ-decomp-coeff}
 \Big\| \Big\{u\big(T,{n}/{N}\big)\Big\}_{n\in \mathbb{Z}^d} \Big\|_{l^2(\mathbb{Z}^d)}\leq C \left( 1+\left(TN^2\right)^{\frac{d}{4}} \right)T^{-\frac{d}{4}}\|u (0,\cdot)\|.
 \end{eqnarray}
   \vskip 5pt
 \noindent (iii) If $T>0$ and $N>0$, then
\begin{align}\label{wang9}
u(T,\cdot) = \sum_{n\in \mathbb{Z}^d} u\big(T, {n}/{N}\big)f_{N,n}(\cdot) \;\;\mbox{in}\;\;L^2(\mathbb{R}^d)\Longleftrightarrow u(0,\cdot)\in \mathscr{P}_N.
\end{align}

\noindent(iv) Given $T>0$, $N>0$,  $\varepsilon\in (0,1)$ and $\{\lambda_n\}_{n\in\mathbb Z^d}\subset \mathbb R^d$, with
$\sup_{n\in \mathbb{Z}^d}|\lambda_n-{n}/{N}|\leq {\varepsilon}/{N}$,
   there is $\widetilde{R}\in L^2(\mathbb R^d)$, with
\begin{eqnarray}\label{wang11}
\big\|\widetilde{R}(\cdot)\big\|\leq C\left(\varepsilon + \left(1+\left(TN^2\right)^{-\frac{1}{2}}\right)
\left(TN^2\right)^{-\frac{d}{4}}e^{-TN^2}\right)\|u(0,\cdot)\|,
\end{eqnarray}
 so that
\begin{eqnarray}\label{wang10}
u(T,\cdot) = \sum_{n\in \mathbb{Z}^d} u(T,\lambda_n)f_{N,n}(\cdot) + \widetilde{R}(\cdot)
\;\;\mbox{in}\;\;L^2(\mathbb{R}^d).
\end{eqnarray}
\end{theorem}
We now give several notes on Theorem \ref{thm-decomp}.
\begin{itemize}
  \item[(a1)] The most important result in this theorem is (\ref{equ-decomp-725}),
  which we call  an asymptotic observability identity for the heat equation \eqref{equ-heat}. Two motivations  to derive (\ref{equ-decomp-725}) are given in order.
 First, it  corresponds to a kind of feedback null approximate controllability, with a cost, for some impulsively controlled heat equations
(see Theorem \ref{theorem4.6}). Second,
it
is  a natural extension of the first identity in (\ref{equ-shannon}).
The later can be explained as follows:
 The first identity in (\ref{equ-shannon}) says that a function $f\in\mathscr{P}_N$ can be completely determined by its values at lattice points $\{\frac{n}{N}\}_{n\in \mathbb{Z}^d}$.
The asymptotic observability identity (\ref{equ-decomp-725}) says that  by observing a solution $u$ of \eqref{equ-heat}
over countably many lattice points $\{\frac{n}{N}\}_{n\in \mathbb{Z}^d}$ in $\mathbb{R}^d$ and at time $T$,
one can approximately recover $u(T,\cdot)$, with the error $R$.

\item [(a2)]
The error $R$ in  \eqref{equ-decomp-725} depends on $T$, $N$ and $u$.
It is a small term when  $N$ is larger. This can be understood in the following way:
  Let $T>0$ and $N>0$. We
   define  three  operators  on $L^2(\mathbb{R}^d)$ via
 $$
 \mathcal{W}_N(u_0):=e^{T\triangle}u_0; \;\; \mathcal{M}_N(u_0):=\sum_{n\in \mathbb{Z}^d}(e^{T\triangle}u_0)(n/N)f_{N,n};\;\; \mathcal{R}_N(u_0):=R_{u_0}\;\;
\mbox{for each}\;\; u_0\in L^2(\mathbb{R}^d),
 $$
  where $R_{u_0}$ is the  error $R$ in \eqref{equ-decomp-725} corresponding to
  the solution of \eqref{equ-heat} with $u(0,\cdot)=u_0(\cdot)$.
  They are clearly linear and bounded operators. We treat $\mathcal{M}_N(u_0)$  as the main part of $\mathcal{W}_N(u_0)$ and $\mathcal{R}_N(u_0)$ as the residual part of $\mathcal{W}_N(u_0)$. Now,
        \eqref{equ-decomp-725} can be
  rewritten as:
   \begin{eqnarray*}\label{wang9111.12}
  \mathcal{W}_N(u_0)=\mathcal{M}_N(u_0)+\mathcal{R}_N(u_0)\;\;\mbox{for each}\;\;
  u_0\in L^2(\mathbb{R}^d).
  \end{eqnarray*}
   According to Proposition 2 in \cite[p. 28]{Stein},
  \begin{eqnarray*}\label{wang9111.13}
  \|\mathcal{W}_N\|_{\mathcal{L}(L^2(\mathbb{R}^d))}=1\;\;\mbox{for all}\;\;N>0.
  \end{eqnarray*}
   When $N$ satisfies that
  \begin{eqnarray}\label{wang9111.11}
 C(d)\left(1+\left(TN^2\right)^{-\frac{d}{4}} \right)e^{-TN^2}\leq \frac{1}{10},
  \end{eqnarray}
 (Here, $C(d)$ is  given in (i) of Theorem \ref{thm-decomp}.)
     we have that
  \begin{eqnarray*}\label{wang9111.14}
  \|\mathcal{R}_N\|_{\mathcal{L}(L^2(\mathbb{R}^d))}\leq \frac{1}{10}.
  \end{eqnarray*}
  From these, we see that if $N$ satisfies \eqref{wang9111.11}, then
  $\|\mathcal{R}_N\|_{\mathcal{L}(L^2(\mathbb{R}^d))}$
  is small, compared with $\|\mathcal{W}_N\|_{\mathcal{L}(L^2(\mathbb{R}^d))}$.

\item [(a3)] The asymptotic observability identity (\ref{equ-decomp-725}) is not true  when the sum on the right hand side of (\ref{equ-decomp-725})
is taken over finitely many lattice  points (see Proposition \ref{prop-cout-ex}).
This motivates us to derive a weak asymptotic observability identity (see
Theorem \ref{thm-asym-ob-balls}) which says
what follows: {\it If it is  known in advance that the initial datum of a solution
$u$ to \eqref{equ-heat} has some kind of decay at infinity, then by observing this solution over finitely many lattice points $\{n/N\}_{n\in \mathbb{Z}^d}\subset B_r(0)$ and at time $T$, one can approximately recover $u(T,\cdot)$, with an error
$R$. Moreover, $R$ tends to zero when $N, r\rightarrow\infty$.}
This weak asymptotic observability identity gives a weak
feedback null approximate controllability, with a cost, for some impulsively controlled heat equations (see Theorem \ref{theorem4.7-added}).

\item [(a4)]
 More general, we build up  a similar
  asymptotic  identity to (\ref{equ-decomp-725})
for  functions in the space $H^s(\mathbb{R}^d)$ with $s>d/2$ (see Theorem \ref{thm-f-d}). Moreover,  we explained that
  $s=d/2$ is critical (see Remark \ref{remark3.2wang}).
  Notice that the Fourier transform of a function in this space decays more slowly than the Fourier transform of  such a function $u(T,\cdot)$ that $u$ solves \eqref{equ-heat} and $T>0$, in general.
   So the asymptotic observable identity \eqref{equ-decomp-725} is a special case of Theorem \ref{thm-f-d}. This will be seen from the proof of Theorem \ref{thm-decomp}.

\item[(a5)]
  The conclusion (ii) gives an estimate
for the sampling values   $\{u(T,\frac{n}{N})\}_{n\in \mathbb{Z}^d}$.
  The conclusion (iii) says that the error term $R(\cdot)$ in (\ref{equ-decomp-725})
cannot be dropped in general.
 The conclusion (iv)  shows some robustness of (\ref{equ-decomp-725}) with respect to the sampling lattice points $\{\frac{n}{N}\}_{n\in \mathbb{Z}^d}$.

\end{itemize}



The rest of this paper is organized as follows: Section 2 proves Theorem \ref{thm-decomp}. Section 3 gives some further studies on the asymptotic observability identity.
Section 4
 presents some applications of the main results
 to some controllability.

\section{Proof of main results }

In the rest of the paper, we use $C(\cdots)$ to denote a  positive constant which depends only on what are enclosed in the brackets and varies in different contexts.
The aim of this section is to show Theorem \ref{thm-decomp}.


\subsection{Some properties on functions in
  Bessel potential spaces}

Consider
 the Bessel potential space
$H^s(\mathbb{R}^d)$ (with $s\in \mathbb{R}$)  equipped with  the norm:
\begin{eqnarray}\label{Wgehngsheng21}
\|f\|_{H^s(\mathbb{R}^d)}:=
\left(\int_{\mathbb{R}^d}\big(1+|\xi|^2\big)^{s}|\hat{f}(\xi)|^2\,\mathrm d\xi\right)^{1/2},\;\; f\in H^s(\mathbb{R}^d).
\end{eqnarray}
For each $s\in \mathbb{R}$, we define a  linear bounded operator
$\langle D\rangle^s: H^s(\mathbb{R}^d)\rightarrow L^2(\mathbb{R}^d)$
in the following manner: Given $f\in H^s(\mathbb{R}^d)$,  let
\begin{eqnarray}\label{wangyuanyuan22}
\widehat{\langle D\rangle^s f}(\xi) := \big(1+|\xi|^2\big)^{\frac{s}{2}}\hat{f}(\xi),~\xi\in \mathbb R^d.
\end{eqnarray}
One can easily see  that
\begin{eqnarray}\label{wangyuanyuan23}
\langle D\rangle^{2m} = (1-\triangle)^m\;\;\mbox{when}\;\; m\in \mathbb{N}^+.
\end{eqnarray}
The main purpose of this subsection is to present the next Theorem \ref{thm-f-d}
 which plays
an important role in the proof of the conclusion (i) in Theorem \ref{thm-decomp}.
\begin{theorem}\label{thm-f-d}
Given  $s>\frac{d}{2}$, there is  $C=C(s,d)$ so that
 any $f\in H^s(\mathbb{R}^d)$ has the property:
 For each $N>0$, there is $R\in L^2(\mathbb R^d)$, with
\begin{eqnarray}\label{wang962.26}
 \|R\|\leq C \left( 1+N^{-\frac{d}{2}} \right)
 \left( \int_{  \xi\in Q_{\pi N}^c  } (1+|\xi|^2)^s |\hat f(\xi)|^2 d\xi \right)^{\frac{1}{2}},
\end{eqnarray}
 so that
 \begin{align}\label{equ-decomp}
f = \sum_{n\in \mathbb{Z}^d} f\left({n}/{N}\right) f_{N,n} + R
\;\;\mbox{in}\;\;
L^2(\mathbb R^d).
\end{align}
\end{theorem}
\begin{remark}\label{remark3.2wang}
(i) If $u$ solves Equation (\ref{equ-heat}) and $T>0$, then for any $s\in \mathbb{R}$,
the function $u(T,\cdot)$ belongs to $H^s(\mathbb{R}^d)$.
So the asymptotic observable identity \eqref{equ-decomp-725} is a special case of the
asymptotic identity \eqref{equ-decomp}. Indeed,
as what we will see in the proof of Theorem \ref{thm-decomp}, the   conclusion (i) in Theorem \ref{thm-decomp} is a direct consequence of Theorem \ref{thm-f-d}.

(ii) From  \eqref{wang962.26}, we see that when
$f\in \mathscr{P}_{N}$, $R$ in \eqref{equ-decomp} disappears, consequently,
\eqref{equ-decomp} is exactly the same as
the first identity in \eqref{equ-shannon}.

 (iii) Theorem \ref{thm-f-d} does not hold for $f\in H^s(\mathbb{R}^d)$ with $s\leq \frac{d}{2}$ in general. The reason is as follows: When $s\leq \frac{d}{2}$, a function in $H^{s}(\mathbb{R}^d)$ may not belong to $C(\mathbb{R}^d)$ and
     may have singularity   in the lattice points $\{\frac{n}{N}\}_{n\in \mathbb{Z}^d}$.
     Such singularity make the coefficients in \eqref{equ-decomp} meaningless. From this point of view,  $s=d/2$ is critical.
 \end{remark}

The next Lemma \ref{lem-comm} serves for the proof of Proposition \ref{lem-smooth-l2-r} which plays a big role in the proof of
the conclusion (i) in Theorem \ref{thm-f-d}.

\begin{lemma}\label{lem-comm}
Let $s>0$  and $\varphi\in C_0^\infty(\mathbb R^d)$.
 Then
\begin{eqnarray}\label{wangyuanyuan25}
\|\varphi f\|_{H^s(\mathbb R^d)}^2
\leq 4^s \left(\|\big((1-\triangle)^{[s]+1}
\varphi\big) f\|^2 +\|\varphi\cdot \langle D \rangle^s f\|^2
\right)\;\;\mbox{for all}\;\;f\in H^s(\mathbb{R}^d),
\end{eqnarray}
where $[s]$ denotes the integer part of $s$, and $\langle D\rangle^s$ is given by  (\ref{wangyuanyuan22}).
\end{lemma}

\begin{proof}
Arbitrarily fix $s>0$, $\varphi\in C_0^\infty(\mathbb R^d)$ and $f\in H^s(\mathbb{R}^d)$.
It follows by (\ref{Wgehngsheng21}) that
\begin{eqnarray}\label{equ-810-1}
\|\varphi f\|^2_{H^s(\mathbb{R}^d)} &=& \int_{\mathbb{R}^d_\xi} \big(1+|\xi|^2\big)^{s} \big|\widehat{\varphi f}(\xi)\big|^2 \,\mathrm d\xi
\nonumber\\
&=&\int_{\mathbb{R}^d_\xi}\big(1+|\xi|^2\big)^s \Big| (2\pi)^{-\frac{d}{2}} \int_{\mathbb{R}^d_\eta} \widehat{\varphi}(\eta) \hat{f}(\xi-\eta)\,\mathrm d\eta \Big|^2 \,\mathrm d\xi.
\end{eqnarray}
Since
\begin{eqnarray*}
\big(1+|\xi|^2\big)^s &\leq& 2^{2s}\left(\big(1+|\eta|^2\big)^s +\big(1+|\xi-\eta|^2\big)^s\right)\nonumber\\
&\leq& 2^{2s}\big(\big(1+|\eta|^2\big)^{2([s]+1)} +\big(1+|\xi-\eta|^2\big)^s\big)
\;\;\mbox{for all}\;\;\xi,\eta\in \mathbb{R}^d,
\end{eqnarray*}
we see from \eqref{equ-810-1}  that
\begin{align}\label{equ-810-3}
\|\varphi f\|^2_{H^s(\mathbb{R}^d)}\leq 2^{2s}(I+II),
\end{align}
where
$$
I := \int_{\mathbb{R}^d_\xi} \Big|
(2\pi)^{-\frac{d}{2}} \int_{\mathbb{R}^d_\eta}  \big(1+|\eta|^2\big)^{[s]+1}
\widehat{\varphi}(\eta)\widehat{f}(\xi-\eta)\,\mathrm d\eta \Big|^2 \,\mathrm d\xi,
$$
$$
II := \int_{\mathbb{R}^d_\xi} \Big| (2\pi)^{-\frac{d}{2}} \int_{\mathbb{R}^d_\eta}
\widehat{\varphi}(\eta)\big(1+|\xi-\eta|^2\big)^{\frac{s}{2}}
\widehat{f}(\xi-\eta)\,\mathrm d\eta \Big|^2 \,\mathrm d\xi.
$$
By the definitions of $I$ and $II$, (\ref{wangyuanyuan22}) and (\ref{wangyuanyuan23}), after some direct computations,
we obtain that
\begin{align}\label{equ-810-4}
I= \|\big((1-\triangle)^{[s]+1}\varphi\big) f\|^2
 \;\;\mbox{and}\;\;
 II= \|\varphi \langle D\rangle^s f\|^2.
\end{align}
Now (\ref{wangyuanyuan25}) follows from  \eqref{equ-810-4} and \eqref{equ-810-3}
at once.
This ends the proof of Lemma \ref{lem-comm}.
\end{proof}

\begin{proposition}\label{lem-smooth-l2-r}
Given $s>\frac{d}{2}$, there is $C=C(s,d)$ so that when  $r>0$,
\begin{align}\label{equ-lpq-625-1}
\Big\|\left\{\| f\|_{C(Q_{r}(rn))}\right\}_{n\in \mathbb{Z}^d}\Big\|_{l^2(\mathbb{Z}^d)}\leq C\left(1+r^{- \frac{d}{2} }\right) \|f\|_{H^s(\mathbb{R}^d)}\;\;\mbox{for all}\;\;f \in H^s(\mathbb{R}^d).
\end{align}
\end{proposition}

\begin{proof}
 Arbitrarily fix $s>\frac{d}{2}$ and $f\in H^s(\mathbb{R}^d)$.
We divide the proof into the following three steps:

\vskip 5pt
\noindent{\it Step 1. We prove that for some $C=C(s,d)$,
\begin{align}\label{wang2.10}
\left\|\{\| f\|_{C(Q_1(n))}\}_{n\in \mathbb{Z}^d}\right\|_{l^2(\mathbb{Z}^d)}\leq C \|f\|_{H^s(\mathbb{R}^d)}.
\end{align}}
Write $s=[s]+\alpha:=m+\alpha$. Let $\varphi: \mathbb{R}^d \rightarrow [0,1]$ be a smooth cutoff function so that
\begin{eqnarray}\label{Wang27}
\varphi=1 \;\;\mbox{over}\;\; Q_1(0)
\;\;\mbox{and}\;\;
\varphi=0  \;\;\mbox{over}\;\;  \mathbb R^d\setminus Q_2(0).
\end{eqnarray}
For each
$n\in \mathbb{Z}^d$, we set
\begin{eqnarray}\label{Wang27-n}
\varphi_n(x) := \varphi(x-n)\;\;\mbox{for each}\;\;x\in \mathbb{R}^d.
\end{eqnarray}
 Since $s>\frac{d}{2}$, we have that
 $H^s(\mathbb{R}^d)  \hookrightarrow C(\mathbb{R}^d) \cap L^\infty(\mathbb R^d)$. Then  there is $C_1(s,d)>0$ so that
\begin{align}\label{equ-715-2}
\|\varphi_nf\|_{C(\mathbb{R}^d)}\leq C_1(s,d)\|\varphi_nf\|_{H^s(\mathbb{R}^d)}
\;\;\mbox{for all}\;\;n\in \mathbb{Z}^d.
\end{align}

We now claim that
\begin{align}\label{equ-810-10}
\sum_{n\in \mathbb{Z}^d}\|f\|^2_{C(Q_1(n))}
&\leq C_1(s,d)^2 2^{2s} \int_{\mathbb{R}^d}\sum_{n\in \mathbb{Z}^d} \Big( \big|\big((1-\triangle)^{m+1}\varphi_n\big) f \big|^2 +|\varphi_n \langle D \rangle^s f|^2 \Big)\,\mathrm dx.
\end{align}
To this end, two facts are given in order. Fact One:  It follows from (\ref{Wang27-n}) and (\ref{Wang27}) that
\begin{align}\label{equ-715-1}
\| f\|_{C(Q_1(n))} \leq \|\varphi_nf\|_{C(\mathbb{R}^d)}\;\;\mbox{for all}\;\;n\in \mathbb{Z}^d;
\end{align}
 Fact Two:  It follows from Lemma \ref{lem-comm} that
\begin{align}\label{equ-716-1}
\|\varphi_n f\|^2_{H^s(\mathbb{R}^d)}  \leq 2^{2s} \Big( \big\| \big((1-\triangle)^{m+1}\varphi_n\big) f \big\|^2 + \|\varphi_n \langle D \rangle^s f\|^2 \Big)\;\;\mbox{for all}\;\;n\in \mathbb{Z}^d.
\end{align}
From \eqref{equ-715-1},
 \eqref{equ-715-2} and \eqref{equ-716-1}, we are led to  \eqref{equ-810-10}.

Next,  according to  the properties of $\varphi_n$ (see (\ref{Wang27-n}) and (\ref{Wang27})), there is $C_2(s)>0$ so that
\begin{align}\label{equ-810-11}
\sum_{n\in \mathbb{Z}^d} \Big( |(1-\triangle)^{m+1}\varphi_n|^2 +|\varphi_n|^2 \Big) \leq C_2(s).
\end{align}

Finally, from  \eqref{equ-810-10} and \eqref{equ-810-11}, we find that
$$
\sum_{n\in \mathbb{Z}^d}\|f\|^2_{C(Q_1(n))}\leq  C_1(s,d)^2 2^{2s} C_2(s)
\int_{\mathbb{R}^d}\left(|f|^2 +|\langle D \rangle^s f|^2\right)\,\mathrm dx,
$$
which, together with (\ref{Wgehngsheng21}) and (\ref{wangyuanyuan22}), leads to (\ref{wang2.10}).

\vskip 5pt
\noindent{\it Step 2. We show (\ref{equ-lpq-625-1}) for the case that  $0<r\leq 1$.}

\noindent Arbitrarily fix  $r\in(0,1]$.  Set $g(x) := f(rx)$, $x\in \mathbb{R}^d$. Then it follows that  $g\in H^s(\mathbb{R}^d)$ and
\begin{align}\label{equ-716-14}
\| f\|^2_{C(Q_r(rn))}  =  \| g\|^2_{C(Q_1(n))}\;\;\mbox{for all}\;\;  n\in \mathbb{Z}^d .
\end{align}
Since $\widehat{g}(\xi) = r^{-d}\widehat{f}(r^{-1}\xi)$, $\xi\in\mathbb R^d$,   we deduce from \eqref{equ-716-14} and (\ref{wang2.10})
 (where $f$ is replaced by $g$) that when $C=C(s,d)$ is given by (\ref{wang2.10}),
\begin{eqnarray*}\label{equ-716-15}
& & \sum_{n\in \mathbb{Z}^d}\| f\|^2_{C(Q_r(rn))}  = \sum_{n\in \mathbb{Z}^d} \| g\|^2_{C(Q_1(n))} \leq C\|g\|^2_{H^s(\mathbb{R}^d)}
\nonumber\\
&=& C  \int_{\mathbb R^d_\xi}  (1+|\xi|^2)^s\big|r^{-d}\widehat{f}(r^{-1}\xi) \big|^2 \mathrm d\xi
=C r^{-d} \int_{\mathbb R^d_\xi}  (1+r^2|\xi|^2)^s \big|\widehat{f}(\xi) \big|^2 \mathrm d\xi
\leq Cr^{-d}\|f\|^2_{H^s(\mathbb{R}^d)}.
\end{eqnarray*}
(In the last inequality in the above, we used the fact that $0<r\leq 1$.)
This proves \eqref{equ-lpq-625-1} for the case that  $0<r\leq 1$.

\vskip 5pt
\noindent{\it Step 3. We show (\ref{equ-lpq-625-1}) for the case that  $r>1$.}

\noindent  We first claim that for any $2\leq m\in \mathbb{N}$ and $\rho>0$,
\begin{align}\label{equ-716-16}
\sum_{n\in \mathbb{Z}^d}\| f\|^2_{C(Q_{m\rho}(m\rho n))}  \leq   3^d\sum_{n\in \mathbb{Z}^d}\| f\|^2_{C(Q_\rho(\rho n))}.
\end{align}
 For this purpose, we arbitrarily fix $2\leq m\in \mathbb{N}$ and $\rho>0$.
 Three observations are given in order.
 \begin{itemize}
   \item[(O1)] Since $f\in  H^s(\mathbb{R}^d)$
 and $s>\frac{d}{2}$, $f$ is a bounded continuous function over $\mathbb{R}^d$.
 Then for each $n\in \mathbb{Z}^d$, there is $x_n^*\in Q_{m\rho}(m\rho n)$
 so that
 \begin{eqnarray*}\label{equ-716-17}
\| f\|_{C(Q_{m\rho}(m\rho n))} = |f(x_n^*)|.
\end{eqnarray*}

   \item[(O2)] For each $n\in \mathbb{Z}^d$, we can choose and then fix a $\beta_n\in \mathbb{Z}^d$ so that
\begin{eqnarray*}\label{equ-716-18}
Q_\rho(\rho \beta_n) \subset Q_{m\rho}(m\rho n))\;\;\mbox{and}\;\;
|f(x^*_n)|=\| f\|_{C(Q_\rho(\rho \beta_n))}.
\end{eqnarray*}
Such $\beta_n$ exists since $2\leq m\in \mathbb{N}$. (It may happen that
when $n\neq n'$, $\beta_n=\beta_{n'}$.)

   \item[(O3)] For each $x\in \mathbb{R}^d$,
there are  at most $3^d$ lattice points $n\in \mathbb{Z}^d$ so that $x\in Q_{m\rho}(m\rho n))$.
Thus, for each $n'\in \mathbb{Z}^d$, the family
$\{Q_\rho(\rho \beta_n)\}_{n\in \mathbb{Z}^d}$  has at most $3^d$ cubes coinciding with  $Q_\rho(\rho n')$.
 \end{itemize}
Then it follows from (O1) and (O2) that
\begin{align}\label{equ-716-19}
\sum_{n\in \mathbb{Z}^d}\| f\|^2_{C(Q_{m\rho}(m\rho n))}
= \sum_{n\in \mathbb{Z}^d}
\| f\|^2_{C(Q_{\rho}(\rho \beta_n))}.
\end{align}
Meanwhile, it follows by  (O3) that
\begin{align}\label{equ-716-20}
\sum_{n\in \mathbb{Z}^d}
\| f\|^2_{C(Q_{\rho}(\rho \beta_n))} \leq 3^d\sum_{n'\in \mathbb{Z}^d}\| f\|^2_{C(Q_\rho(\rho n'))}.
\end{align}
Combining \eqref{equ-716-19} and \eqref{equ-716-20} leads to the claim \eqref{equ-716-16}.

We now   arbitrarily fix   $r> 1$. One can always find $2\leq m\in \mathbb{N}^+$ and $\rho\in (\frac{1}{2},1]$ so that $r=m\rho$. Then by \eqref{equ-716-16}, we find that
\begin{align}\label{equ-716-21}
\sum_{n\in \mathbb{Z}^d}\| f\|^2_{C(Q_{r}(r n))} = \sum_{n\in \mathbb{Z}^d}\| f\|^2_{C(Q_{m\rho}(m\rho n))}  \leq   3^d\sum_{n\in \mathbb{Z}^d}\| f\|^2_{C(Q_\rho(\rho n))}.
\end{align}
Meanwhile, since $\rho \in (\frac{1}{2},1]$,
it follows from Step 2 that (\ref{equ-lpq-625-1}) holds for $r=\rho$. Thus we have that
\begin{align}\label{equ-716-22}
\sum_{n\in \mathbb{Z}^d}\| f\|^2_{C(Q_\rho(\rho n))}
\leq C\big(1+\rho^{-\frac{d}{2}}\big)\|f\|^2_{H^s(\mathbb{R}^d)}\leq C\big(1+2^{\frac{d}{2}}\big)\|f\|^2_{H^s(\mathbb{R}^d)}.
\end{align}
By \eqref{equ-716-21} and \eqref{equ-716-22},
 we find that
 $$
 \sum_{n\in \mathbb{Z}^d}\| f\|^2_{C(Q_{r}(r n))}
 \leq C3^d\big(1+2^{\frac{d}{2}}\big)\|f\|^2_{H^s(\mathbb{R}^d)}:=\hat C(d,s)\|f\|^2_{H^s(\mathbb{R}^d)}\leq \hat C(d,s)\big(1+r^{-\frac{d}{2}}\big)\|f\|^2_{H^s(\mathbb{R}^d)},
 $$
 which leads to  \eqref{equ-lpq-625-1} for the case that $r> 1$.  This ends the proof of Proposition \ref{lem-smooth-l2-r}.
\end{proof}

We now on the position to prove  Theorem \ref{thm-f-d}.


\begin{proof}[Proof of Theorem \ref{thm-f-d}]
  Arbitrarily fix $N>0$.
 Define two operators $\chi_{\leq N}(D)$ and $\chi_{> N}(D)$ on $L^2(\mathbb R^d)$
 in the following manner: For each $g\in L^2(\mathbb R^d)$, set
\begin{eqnarray}\label{aug-chi-N-c-11}
\mathcal F(\chi_{\leq N}(D)g)(\xi) :=\chi_{Q_{\pi N}}(\xi) \mathcal F(g)(\xi)\;\;\mbox{and}\;\;
\mathcal F(\chi_{> N}(D)g)(\xi):= \chi_{Q_{\pi N}^c}(\xi) \mathcal F(g)(\xi),~\xi\in \mathbb{R}^d,
\end{eqnarray}
where $Q_{\pi N}^c$ denotes the complementary set of $Q_{\pi N}$ in $\mathbb R^n$.
 Then, we arbitrarily fix $s>\frac{d}{2}$ and  $f\in H^s(\mathbb{R}^d)$.
We organize the rest of the proof by two steps.

\vskip 5pt
\noindent {\it Step 1. We make  decompositions on $f$. }
\vskip 5pt
First of all, by \eqref{aug-chi-N-c-11}, we can directly check that
\begin{align}\label{equ-714-1}
f=\chi_{\leq N}(D)f+\chi_{> N}(D)f\;\;\mbox{in}\;\;
H^s(\mathbb{R}^d).
\end{align}
In particular, we have that
\begin{eqnarray}\label{yubiao9152.29}
f, \;\chi_{\leq N}(D)f,\;\chi_{> N}(D)f\in H^s(\mathbb{R}^d).
\end{eqnarray}
By \eqref{equ-714-1} and the imbedding
$H^s(\mathbb{R}^d)  \hookrightarrow C(\mathbb{R}^d)$ (which follows from the assumption that $s>\frac{d}{2}$), we have that
\begin{eqnarray}\label{GGwang9132.27}
f(n/N)=(\chi_{\leq N}(D)f)(n/N)+(\chi_{> N}(D)f)(n/N)\;\;\mbox{for all}\;\;
n\in\mathbb{Z}^d.
\end{eqnarray}

From (\ref{yubiao9152.29}), we can use Proposition \ref{lem-smooth-l2-r} to see that
\begin{eqnarray*}
\{f(n/N)\}_{n\in \mathbb{Z}^d},\;\{(\chi_{\leq N}(D)f)(n/N)\}_{n\in \mathbb{Z}^d},\;\{(\chi_{> N}(D)f)(n/N)\}_{n\in \mathbb{Z}^d}
\in l^2(\mathbb{Z}^d).
\end{eqnarray*}
These, along with the conclusion (iii) in Lemma \ref{lem-orth}, imply that
\begin{eqnarray}\label{yubiao9152.31}
\sum_{n\in \mathbb{Z}^d}f(n/N)f_{N,n},\;
\sum_{n\in \mathbb{Z}^d}(\chi_{\leq N}f)(n/N)f_{N,n},\;
\sum_{n\in \mathbb{Z}^d}(\chi_{>N}f)(n/N)f_{N,n}\in L^2(\mathbb{R}^d).
\end{eqnarray}
From \eqref{yubiao9152.31} and \eqref{yubiao9152.29}, we can define
\begin{align}\label{equ-716-fd-0-1}
R_N := - \sum_{n\in \mathbb{Z}^d} \big( \chi_{> N}(D)f \big)\left({n}/{N} \right)  f_{N,n} +\chi_{> N}(D)f\;\;\mbox{in}\;\; L^2(\mathbb{R}^d).
\end{align}
From \eqref{GGwang9132.27} and \eqref{yubiao9152.31}, one can easily obtain that
\begin{eqnarray}\label{GGwang9132.33}
\sum_{n\in \mathbb{Z}^d}f(n/N)f_{N,n}=
\sum_{n\in \mathbb{Z}^d} \big( \chi_{\leq N}(D)f \big)
\left( {n}/{N} \right) f_{N,n}
+
\sum_{n\in \mathbb{Z}^d}(\chi_{>N}(D)f)(n/N)f_{N,n}\;\;\mbox{in}\;\;L^2(\mathbb{R}^d).
\end{eqnarray}

Meanwhile, by \eqref{aug-chi-N-c-11} and \eqref{wang1}, we have that
 $\chi_{\leq N}(D)f\in \mathscr{P}_N$. Thus, we can use
 Theorem \ref{lem-shannon}, where $f$ is replaced by $\chi_{\leq N}(D)f$, to see that
\begin{align}\label{equ-714-2}
\chi_{\leq N}(D)f=\sum_{n\in \mathbb{Z}^d} \big( \chi_{\leq N}(D)f \big)
\left( {n}/{N} \right) f_{N,n}
\;\;\mbox{in}\;\;
L^2(\mathbb{R}^d).
\end{align}

Finally, by \eqref{equ-716-fd-0-1}, \eqref{GGwang9132.33}, \eqref{equ-714-2}
and \eqref{equ-714-1}, we obtain that
\begin{align}\label{equ-714-3}
f=\sum_{n\in \mathbb{Z}^d} f\left({n}/{N} \right) f_{N,n}+ R_N \;\;\mbox{in}\;\;
L^2(\mathbb{R}^d).
\end{align}

\vskip 5pt

\noindent {\it Step 2. We estimate two terms on the right hand side of \eqref{equ-716-fd-0-1}.}
\vskip 5pt

For the second one, we use the Parseval-Plancherel formula and (\ref{aug-chi-N-c-11})  to see that
\begin{eqnarray}\label{equ-716-fd-1}
\|\chi_{> N}(D)f\|
=  \|\chi_{Q_{\pi N}^c} \widehat{f}\|
\leq
\left( \int_{ \xi\in Q_{\pi N}^c } (1+|\xi|^2)^{s} |\hat{f}(\xi)|^2 \mathrm d\xi \right)^{1/2}.
\end{eqnarray}
Next we deal with the term: $-\sum_{n\in \mathbb{Z}^d} \big( \chi_{> N}(D)f \big)\left({n}/{N} \right)  f_{N,n}$. By (iii) of Lemma \ref{lem-orth} in Appendix, we find that
\begin{eqnarray}\label{equ-716-fd-2}
\big\|\sum_{n\in \mathbb{Z}^d} \big( \chi_{> N}(D)f \big) \left({n}/{N} \right) f_{N,n} \big\|
&=& N^{-d/2} \big\| \big\{ \big( \chi_{> N}(D)f \big) \left({n}/{N} \right) \big\}_{n\in \mathbb{Z}^d} \big\|_{l^2(\mathbb{Z}^d)}
\nonumber\\
&\leq& N^{-d/2} \big\| \big\{  \| \chi_{> N}(D)f\|_{C(Q_{\frac{1}{N}}({n}/{N}))} \big\}_{n\in \mathbb{Z}^d} \big\|_{l^2(\mathbb{Z}^d)}.
\end{eqnarray}
Since $s>\frac{d}{2}$, we can apply Proposition \ref{lem-smooth-l2-r}, where $(s,f,r)$ is replaced by $(s,\chi_{> N}(D)f,{1}/{N})$,   to find   $C(s,d)>0$ so that
\begin{align}\label{equ-716-fd-3}
\Big\| \Big\{\| \chi_{> N}(D)f\|_{C(Q_{\frac{1}{N}}(\frac{n}{N}))} \Big\}_{n\in \mathbb{Z}^d} \Big\|_{l^2(\mathbb{Z}^d)}
\leq C(s,d)(1+N^{d/2})\|\chi_{> N}(D)f\|_{ H^{s}(\mathbb R^d) }.
\end{align}
Now, it follows from \eqref{equ-716-fd-2}-\eqref{equ-716-fd-3} that
\begin{eqnarray}\label{equ-716-fd-5}
\Big\| \sum_{n\in \mathbb{Z}^d} \big( \chi_{> N}(D)f \big) \left({n}/{N}\right) f_{n,N} \Big\|
 \leq  C(s,d)\left(1+N^{-d/2}\right)
  \left( \int_{ \xi\in Q_{\pi N}^c } \left(1+\pi^2 |\xi|^2\right)^{s} |\hat{f}(\xi)|^2 \mathrm d\xi \right)^{1/2}.
\end{eqnarray}
\vskip 5pt
\noindent {Step 3. We finish the proof.}
\vskip 5pt
From  \eqref{equ-716-fd-0-1}, \eqref{equ-716-fd-1} and (\ref{equ-716-fd-5}), we see that
\begin{eqnarray*}\label{equ-716-fd-6}
\|R_N\|
 &\leq& \big( C(s,d)+1 \big) \left(1+N^{-d/2}\right)
\Big( \int_{ \xi\in Q_{\pi N}^c } \left(1+|\xi|^2\right)^{s} |\hat{f}(\xi)|^2 \mathrm d\xi \Big)^{1/2}.
\end{eqnarray*}
This, together with (\ref{equ-714-3}), leads to
\eqref{wang962.26} and
 \eqref{equ-decomp}. Hence, we end the proof of Theorem \ref{thm-f-d}.
\end{proof}

\subsection{Some estimates on solutions of the heat equation}

This subsection presents one lemma
on properties of the heat equation.
 It will be used in the proof of the conclusion (ii)  in Theorem \ref{thm-decomp}.
  It   will also be used in the proof of Proposition \ref{cor-purb-2} (in the next subsection) which will play an important role in the proof of the conclusion (iv) in Theorem \ref{thm-decomp}.

\begin{lemma}\label{lem-lpq}
 There is $C=C(d)$ so that for each solution  $u$ to \eqref{equ-heat},  each $r>0$ and each $T>0$,
\begin{eqnarray}
\left\|\{\| u(T,\cdot)\|_{C(Q_{r}(rn))}\}_{n\in \mathbb{Z}^d}\right\|_{l^2(\mathbb{Z}^d)}
&\leq&   C\big( 1+(Tr^{-2})^{\frac{d}{4}} \big)
T^{-\frac{d}{4}} \|u(0,\cdot)\|;
\label{equ-lpq-1}\\
\left\|\{\|\nabla u(T,\cdot)\|_{C(Q_{r}(rn))}\}_{n\in \mathbb{Z}^d}\right\|_{l^2(\mathbb{Z}^d)}
&\leq& C\big( 1+(Tr^{-2})^{\frac{d}{4}} \big)
T^{-\frac{d}{4}-\frac{1}{2}} \|u(0,\cdot)\|;
\label{equ-lpq-2}\\
\left\|\left\{ u(T,rn)\right\}_{n\in \mathbb{Z}^d}\right\|_{l^2(\mathbb{Z}^d)}
&\leq& C\left(1+(Tr^{-2})^{\frac{d}{4}}\right)T^{-\frac{d}{4}}\|u(0,\cdot)\|.
\label{equ-coeff-bound-1}
\end{eqnarray}
\end{lemma}

\begin{proof} Arbitrarily fix $T>0$, $r>0$ and a solution $u$ to \eqref{equ-heat}.
Let $u_0(x):=u(0,x)$, $x\in \mathbb{R}^d$. Set
 $$
 v_0(x):=u_0\big(\sqrt{T}x\big)\;\;\mbox{for a.e.}\;\;x\in \mathbb{R}^d.
 $$
We first show (\ref{equ-lpq-1}). To this end,
we let
\begin{eqnarray}\label{wang9102.39}
w(x):=\left(e^{\triangle}v_0\right)\Big(\frac{x}{\sqrt{T}}\Big)\;\;\mbox{for each}\;\; x\in \mathbb{R}^d.
\end{eqnarray}
We claim the following three facts:

  \noindent Fact One: For each $\alpha\in \mathbb{N}^d$, there is $C(\alpha,d)>0$ so that
\begin{align}\label{equ-lpq-3}
\|D^\alpha_x u(T,\cdot)\|\leq C(\alpha,d) T^{-\frac{|\alpha|}{2}}\|u_0\|.
\end{align}

 \noindent Fact Two: We have that
\begin{align}\label{equ-7-1-1}
u(T,x) =\left(e^{T\triangle}u_0\right)(x)= w(x) \;\;\mbox{for each}\;\;x\in \mathbb{R}^d.
\end{align}

\noindent  Fact Three:  There is $C=C(d)$ so that
\begin{eqnarray}\label{equ-7-1-2}
&&\Big\|\left\{\|w\|_{C(Q_r(nr))}\right\}_{n\in \mathbb{Z}^d}\Big\|_{l^2(\mathbb{Z}^d)}
\leq  C\Big(1+\big(Tr^{-2}\big)^{\frac{d}{4}}\Big)T^{-\frac{d}{4}}\|u_0\|.
\end{eqnarray}

To show (\ref{equ-lpq-3}), we use  the Parseval-Plancherel formula to see that
\begin{align}\label{equ-lpq-7.5}
\left\|D^\alpha u(T,\cdot)\right\| =
\left( \int_{\mathbb R^d_\xi} \big|
 \xi^\alpha e^{-T|\xi|^2}\widehat{u_0}(\xi) \big|^2 \mathrm d\xi
 \right)^{1/2}
\leq
\left(\frac{|\alpha|}{2T}\right)^{\frac{|\alpha|}{2}}
e^{-\frac{|\alpha|}{2}}\left\|\widehat{u_0}\right\|
=\left(\frac{|\alpha|}{2T}\right)^{\frac{|\alpha|}{2}}
e^{-\frac{|\alpha|}{2}}\|u_0\|,
\end{align}
which leads to \eqref{equ-lpq-3}.

To show \eqref{equ-7-1-1}, we use the heat kernel to see that
\begin{align}\label{equ-7-1-3}
u(T,x)=\left(e^{T\triangle}u_0\right)(x) =\frac{1}{(4\pi T)^{d/2}}\int_{\mathbb{R}^d}e^{-\frac{|x-y|^2}{4T}}u_0(y)\,\mathrm dy, \quad x\in \mathbb{R}^d.
\end{align}
Changing variable $y \mapsto \sqrt{T}y$ in \eqref{equ-7-1-3} and using
 \eqref{wang9102.39}, we find that
$$
u(T,x)=\frac{1}{(4\pi)^{d/2}}\int_{\mathbb{R}^d}
e^{-\frac{\left|\frac{x}{\sqrt{T}}-y\right|^2}{4}}u_0\left(\sqrt{T}y\right)\,\mathrm dy = w(x)\;\;\mbox{for any}\;\; x\in \mathbb{R}^d.
$$
 This  gives \eqref{equ-7-1-1}.

To prove \eqref{equ-7-1-2}, we first apply Proposition \ref{lem-smooth-l2-r}, where
$(s,r,f)$ is taken as $(d, \frac{r}{\sqrt{T}},e^{\triangle}v_0)$, to obtain that
$$
\Big\|\Big\{\left\|e^{\triangle}v_0
   \right\|_{C(Q_{\frac{r}{\sqrt{T}}}
  (n\frac{r}{\sqrt{T}}))}\Big\}_{n\in \mathbb{Z}^d}\Big\|_{l^2(\mathbb{Z}^d)}
  \leq
  C\left(1+(Tr^{-2})^{\frac{d}{4}} \right)
 \left\|e^{\triangle}v_0 \right\|_{H^d(\mathbb{R}^d)}.
  $$
  Here $C$ is given by Proposition \ref{lem-smooth-l2-r}.
 Then from the above,  \eqref{wang9102.39} and
 \eqref{equ-lpq-3} (where $T=1$ and $|\alpha|\leq d$), we can find  some $C'(d)>0$
 so that
\begin{eqnarray*}
  \left\|\left\{\|w\|_{C(Q_r(nr))}\right\}_{n\in \mathbb{Z}^d}\right\|_{l^2(\mathbb{Z}^d)}
  \leq C'(d)\left(1+(Tr^{-2})^{\frac{d}{4}} \right)\|v_0 \|
=C'(d)\left(1+(Tr^{-2})^{\frac{d}{4}} \right)T^{-\frac{d}{4}}\|u_0\|.
\end{eqnarray*}
 This gives \eqref{equ-7-1-2}.

Now, it follows by \eqref{equ-7-1-1},  \eqref{wang9102.39} and \eqref{equ-7-1-2} that
\begin{align*}
\left\|\{\| u(T,\cdot)\|_{C(Q_{r}(rn))}\}_{n\in \mathbb{Z}^d}\right\|_{l^2(\mathbb{Z}^d)}&= \Big\|\Big\{\Big\| \left(e^{\triangle}v_0\right)
\Big(\frac{\cdot}{\sqrt{T}}\Big)\Big\|_{C(Q_{r}(rn))}\Big\}_{n\in \mathbb{Z}^d}\Big\|_{l^2(\mathbb{Z}^d)}\\
&\leq C\left(1+(Tr^{-2})^{\frac{d}{4}}\right) T^{-\frac{d}{4}}\|u_0\|,
\end{align*}
which leads to \eqref{equ-lpq-1}.

We next show  \eqref{equ-lpq-2}. Arbitrarily fix $j\in\{1,\cdots,d\}$. Notice that $\partial_j u$ is the solution to  the heat equation:
$\partial_t v=\triangle v$ over $(0,\infty)\times \mathbb{R}^d$ with the initial condition: $v(0,\cdot)=\partial_j u_0(\cdot)\in H^{-1}(\mathbb{R}^d)$.
Thus, by the smooth effect of the heat equation and by the standard translation argument, we can
 apply \eqref{equ-lpq-1}, where $(u(\cdot.\cdot),u_0(\cdot),T)$ is replaced by $(\partial_j u(T/2+\cdot, \cdot), \partial_j u(T/2,\cdot),{T}/{2})$, to get that
\begin{eqnarray*}
\left\|\{\| \partial_j u(T,\cdot)\|_{C(Q_{r}(rn))}\}_{n\in \mathbb{Z}^d}\right\|_{l^2(\mathbb{Z}^d)}
\leq C\left(1+(Tr^{-2})^{\frac{d}{4}}\right)
(T/2)^{-\frac{d}{4}}\|\partial_j u(T/2,\cdot)\|.
\end{eqnarray*}
This yields that
\begin{eqnarray*}
& & \sum_{n\in \mathbb Z^d} \| \nabla u(T,\cdot)\|_{C(Q_{r}(rn))}^2
\leq \sum_{n\in \mathbb Z^d} \sum_{j=1}^d \| \partial_j u(T,\cdot)\|_{C(Q_{r}(rn))}^2
\nonumber\\
&\leq&  C^2\left(1+(Tr^{-2})^{\frac{d}{4}}\right)^2
(T/2)^{-\frac{d}{2}} \sum_{j=1}^d\|\partial_j u(T/2,\cdot)\|^2,
\end{eqnarray*}
which, together with  \eqref{equ-lpq-3}, leads to \eqref{equ-lpq-2}.

Finally, \eqref{equ-coeff-bound-1} follows directly from  \eqref{equ-lpq-1}.
This ends the proof of
Lemma \ref{lem-lpq}.
\end{proof}


\subsection{Stability of some functions on sampling lattice points}
This section gives some stability estimates for some functions on sampling lattice points. The next Lemma \ref{lem-conti} serves for the proof of Corollary \ref{cor-purb-1}, while the later  will be used in the proof of Proposition \ref{cor-purb-2} which  plays a big role in the proof of the conclusion (iv) in Theorem \ref{thm-decomp}. Recall (\ref{wang1}) for the definition of $\mathscr{P}_N$.

\begin{lemma}\label{lem-conti}
Let $B$ and $L$ be two positive constants. Let
$\{\lambda_n\}_{n\in \mathbb{Z}^d}\subset\mathbb{R}^d$ satisfy that
\begin{align}\label{equ-726-1}
\sup_{h\in \mathscr{P}_{1},\|h\|\leq 1} \Big( \sum_{n\in \mathbb{Z}^d} |h(\lambda_n)|^2 \Big)
\leq B.
\end{align}
Assume that  $\{\mu_n\}_{n\in \mathbb{Z}^d} \subset \mathbb{R}^d$ satisfies that
$\sup_{n\in\mathbb{Z}^d}|\lambda_n-\mu_n|\leq L$.
Then for every $f\in \mathscr{P}_{1}$,
$$
\sum_{n\in \mathbb{Z}^d} |f(\lambda_n)-f(\mu_n)|^2
\leq B(e^{\pi d L}-1)^2\|f\|^2.
$$
\end{lemma}

\begin{proof}
For the case that $d=1$, Lemma \ref{lem-conti} was proved in \cite[Lemma 3, p. 181]{Y}. For the high-dimensional case, the proof is very  similar, provided that one uses the Taylor expansion in several variables, instead of  one variable. We omit the details. This ends the proof of Lemma \ref{lem-conti}.
\end{proof}

\begin{corollary}\label{cor-purb-1}
Let $N>0$ and $\varepsilon\in (0,1)$.  Assume that $\{\lambda_n\}_{n\in \mathbb{Z}}\subset \mathbb R^d$ satisfies that
$\sup_{n\in \mathbb{Z}^d}|\lambda_n-{n}/{N}|\leq {\varepsilon}/{N}$.
Then for every $f\in \mathscr{P}_{N}$,
\begin{eqnarray}\label{wang952.35}
\big\| \big\{ f(\lambda_n)-f\big({n}/{N}\big) \big\}_{n\in \mathbb{Z}^d} \big\|_{l^2(\mathbb{Z}^d)}
\leq \varepsilon \pi d e^{\pi d} N^{\frac{d}{2}} \|f\| .
\end{eqnarray}
\end{corollary}

\begin{proof}
Arbitrarily fix $N$, $\varepsilon$,  $\{\lambda_n\}_{n\in \mathbb{Z}}$ and $f$ as required.
 Set $g(x):=f(N^{-1}x)$, $x\in \mathbb{R}^d$. Then one can directly check that
$g\in \mathscr{P}_{1}$.
 Meanwhile, by the second equality in \eqref{equ-shannon} in Theorem \ref{lem-shannon}, we find that
\begin{eqnarray*}
\int_{\mathbb{R}^d}|h(x)|^2\,\mathrm dx = \sum_{n\in \mathbb{Z}^d} \big|h(n)\big|^2
\;\;\mbox{for all}\;\;h\in \mathscr{P}_{1}.
\end{eqnarray*}
From this, we see that
\begin{align}\label{wang952.36}
\sup_{h\in \mathscr{P}_{1},\|h\|\leq 1} \Big( \sum_{n\in \mathbb{Z}^d} |h(n)|^2 \Big)
=1.
\end{align}
Since $g\in \mathscr{P}_{1}$ and because of  \eqref{wang952.36},  we can use Lemma \ref{lem-conti}, where $(f,B,L,\lambda_n,\mu_n)$ is replaced by $(g,1,\varepsilon,n,N\lambda_n)$,  to obtain that
\begin{eqnarray*}\label{equ-726-3}
\left\|\{g(N\lambda_n)-g(n)\}_{n\in \mathbb{Z}^d}\right\|_{l^2(\mathbb{Z}^d)}
= \Big( \sum_{n\in \mathbb{Z}^d} |g(N\lambda_n)-g(n)|^2 \Big)^{1/2}
\leq (e^{\pi d \varepsilon}-1)\|g\|.
\end{eqnarray*}
Because $0\leq e^s-1\leq se^s$ for all $s\geq 0$, the above  indicates that
\begin{eqnarray*}\label{equ-726-5}
\|\{g(N\lambda_n)-g(n)\}_{n\in \mathbb{Z}^d}\|_{l^2(\mathbb{Z}^d)}\leq \pi d \varepsilon e^{\pi d \varepsilon} \|g\|\leq \pi d e^{\pi d} \varepsilon\|g\|,
\end{eqnarray*}
which, along with the definition of $g$,
leads to \eqref{wang952.35}. This ends the proof of Corollary \ref{cor-purb-1}.
\end{proof}

\begin{proposition}\label{cor-purb-2}
There is $C(d)$ so that any solution  $u$  to Equation \eqref{equ-heat}
has the property: If  $T>0$, $N>0$, $\varepsilon\in (0,1)$ and
 $\{\lambda_n\}_{n\in\mathbb Z^d}\subset\mathbb R^d$ satisfies that
 \begin{eqnarray}\label{equ-726-7}
 \sup_{n\in \mathbb{Z}^d}|\lambda_n-{n}/{N}|\leq {\varepsilon}/{N},
 \end{eqnarray}
then
\begin{eqnarray}
\big\| \big\{ u(T,\lambda_n)-u\big(T,{n}/{N}\big) \big\}_{n\in \mathbb{Z}^d} \big\|_{l^2(\mathbb{Z}^d)}
\leq C \varepsilon N^{\frac{d}{2}}
\Big( 1+(TN^2)^{-\frac{d}{4}-\frac{1}{2}} e^{-TN^2} \Big)
\|u(0,\cdot)\|.
\label{wang952.48}
\end{eqnarray}

%
%
\end{proposition}

\begin{proof}
Arbitrarily fix $T>0$, $N>0$, $\varepsilon\in (0,1)$, $\{\lambda_n\}_{n\in\mathbb Z^d}$ and   $u$  as required. Let $\chi_{\leq N}(D)$ and $\chi_{>N}(D)$ be the  operators defined in (\ref{aug-chi-N-c-11}) (in the proof of Theorem \ref{thm-f-d}).  Set $u_0(\cdot):= u(0,\cdot)$.  Then it follows that
\begin{align}\label{equ-726-7-1}
u(T,\cdot) = \big(e^{T\triangle} u_0\big)(\cdot)
 = \big(e^{T\triangle}\chi_{\leq  N}(D)u_0\big)(\cdot)
 +  \big(e^{T\triangle}\chi_{>  N}(D)u_0\big)(\cdot)\;\;\mbox{in}\;\;L^2(\mathbb{R}^d).
\end{align}

We now claim that there exists $C_1(d)>0$ so that
\begin{eqnarray}\label{equ-726-10}
 & & \big\| \big\{ (e^{T\triangle}\chi_{>  N}(D)u_0)(\lambda_n)
 - (e^{T\triangle}\chi_{>  N}(D)u_0)({n}/{N}) \big\}_{n\in \mathbb{Z}^d} \big\|_{l^2(\mathbb{Z}^d)}
 \nonumber\\
&\leq&  C_1(d) \varepsilon N^{\frac{d}{2}}
(TN^2)^{-\frac{d}{4}-\frac{1}{2}}e^{-TN^2}\|u_0\|.~~
\end{eqnarray}
Indeed, since $e^{T\triangle}\chi_{>  N}(D)u_0\in C^1(\mathbb{R}^d)$, we can use    the mean value theorem and (\ref{equ-726-7}) to  find that
\begin{align*}
\Big| (e^{T\triangle}\chi_{>  N}(D)u_0)(\lambda_n)
- (e^{T\triangle}\chi_{>  N}(D)u_0)\left({n}/{N}\right) \Big|
\leq \varepsilon N^{-1}  \sup_{x\in Q_{\frac{1}{N}}(\frac{n}{N}) }\left|\nabla (e^{T\triangle}\chi_{>  N}(D)u_0)(x)\right|,
\end{align*}
from which, it follows that
\begin{eqnarray*}
& & \| \left\{(e^{T\triangle}\chi_{>  N}(D)u_0)(\lambda_n)
-(e^{T\triangle}\chi_{>N}(D)u_0)\left({n}/{N}\right) \right\}_{n\in \mathbb{Z}^d} \big\|_{l^2(\mathbb{Z}^d)}
\nonumber\\
&\leq& \varepsilon N^{-1}  \Big\| \Big\{ \|\nabla e^{T\triangle}\chi_{>  N}(D)u_0\|_{C(Q_{\frac{1}{N}}(\frac{n}{N}))} \Big\}_{n\in \mathbb{Z}^d} \Big\|_{l^2(\mathbb{Z}^d)}.
\end{eqnarray*}
This, together with    \eqref{equ-lpq-2} in Lemma \ref{lem-lpq}, where $(r,T,u(0,\cdot))$ is replaced by $(1/N,T/2,e^{T\triangle/2}\chi_{> N}(D)u_0)$,  yields    that
for some $C(d)>0$,
\begin{eqnarray}\label{equ-726-9}
& & \big\| \big\{(e^{T\triangle}\chi_{>  N}(D)u_0)(\lambda_n)
- (e^{T\triangle}\chi_{>  N}(D)u_0)\left({n}/{N}\right) \big\}_{n\in \mathbb{Z}^d} \big\|_{l^2(\mathbb{Z}^d)}
\nonumber\\
&\leq& C(d)  \varepsilon  N^{-1}  \Big( 1+(TN^2)^{\frac{d}{4}} \Big) T^{-\frac{d}{4}-\frac{1}{2}}\|e^{T\triangle/2}\chi_{> N}(D)u_0\|.
\end{eqnarray}
Meanwhile, we clearly have that for some $C(d)>0$,
$$
\left\|e^{T\triangle/2}\chi_{>  N}(D)u_0\right\|\leq e^{-\frac{T}{2}(\pi N)^2}\|u_0\|
\;\;\mbox{and}\;\;
1+(TN^2)^{\frac{d}{4}} \leq C(d)e^{(\frac{\pi^2}{2}-1)TN^2}.
$$
These, along with  \eqref{equ-726-9}, lead to  the  claim (\ref{equ-726-10}).

%
Next, by the definition of $\chi_{\leq N}(D)$ (see (\ref{aug-chi-N-c-11})), we find that $e^{T\triangle}\chi_{\leq  N}(D)u_0 \in \mathscr{P}_N$.
From this and \eqref{equ-726-7}, we can  apply Corollary \ref{cor-purb-1},
where $f$ is replaced by $e^{T\triangle}\chi_{\leq  N}(D)u_0$, to find that
\begin{eqnarray}\label{equ-726-8}
& & \big\| \big\{(e^{T\triangle}\chi_{\leq  N}(D)u_0)(\lambda_n)
- (e^{T\triangle}\chi_{\leq  N}(D)u_0)\left({n}/{N}\right) \big\}_{n\in \mathbb{Z}^d} \big\|_{l^2(\mathbb{Z}^d)}
\nonumber\\
&\leq& \pi d e^{\pi d}\varepsilon N^{\frac{d}{2}}\|e^{T\triangle}\chi_{\leq  N}(D)u_0\| \leq \pi d e^{\pi d}\varepsilon N^{\frac{d}{2}}\|u_0\|.
\end{eqnarray}

Finally, it follows from \eqref{equ-726-7-1},  \eqref{equ-726-8} and \eqref{equ-726-10} that
\begin{eqnarray*}
 \big\| \big\{u(T,\lambda_n)-u\big(T,{n}/{N}\big) \big\}_{n\in \mathbb{Z}^d} \big\|_{l^2(\mathbb{Z}^d)}
\leq   \big( \pi d e^{\pi d} + C_1(d) \big) \varepsilon
 N^{\frac{d}{2}}
\Big( 1+(TN^2)^{-\frac{d}{4}-\frac{1}{2}}e^{-TN^2} \Big)
\|u_0\|,
\end{eqnarray*}
which leads to \eqref{wang952.48}.
This completes the proof of Proposition \ref{cor-purb-2}.
\end{proof}

\subsection{Proof of Theorem \ref{thm-decomp}}

We are now in the position to show Theorem \ref{thm-decomp}.

\begin{proof}[Proof of Theorem \ref{thm-decomp}]
Let $u$ be a solution to \eqref{equ-heat}.
 We will show the conclusions (i)-(iv) one by one.

\vskip 5pt
(i) Arbitrarily fix $T>0$ and $N>0$. Set
\begin{eqnarray}\label{sep-tj-thm1-i-pf-1}
 v_0(x):=u\left(0,\sqrt{T}x\right)\;\;\mbox{for a.e.}\;\;x\in\mathbb R^d.
\end{eqnarray}
Since $e^{\triangle}v_0 \in H^d(\mathbb R^d)$,  we can apply Theorem \ref{thm-f-d}, where $(s,N,f)$ is replaced by $(d,\sqrt{T}N,e^{\triangle}v_0)$, to find  $R_1\in L^2(\mathbb R^d)$ and $C(d)$, with
\begin{eqnarray}\label{sep-tj-thm1-i-pf-2}
 \|R_1\| \leq C(d) \Big( 1+ (TN^2)^{-d/4} \Big)
  \Big( \int_{\xi \in Q_{\pi \sqrt{T}N}^c} (1+|\xi|^2)^d | \mathcal F( e^{\triangle}v_0)(\xi)|^2 \mathrm d\xi \Big)^{1/2},
\end{eqnarray}
 so that
\begin{eqnarray}\label{sep-tj-thm1-i-pf-3}
 e^\Delta v_0 = \sum_{n\in\mathbb Z^d} \left(e^\Delta v_0\right)\Big(\frac{n}{\sqrt{T}N}\Big) f_{\sqrt{T}N,n}+R_1\;\;\;\mbox{in}\;\; L^2(\mathbb R^d).
\end{eqnarray}
We first  use \eqref{sep-tj-thm1-i-pf-3} to show that
\begin{eqnarray}\label{gwang9122.61}
 u(T,\cdot)
 = \sum_{n\in\mathbb Z^d} u(T,{n}/{N}) f_{N,n}(\cdot)
 + R(\cdot)\;\;\mbox{in}\;\;L^2(\mathbb R^d),
\end{eqnarray}
where $R$ is defined by
\begin{eqnarray}\label{wang9112.62}
R(x):= R_1\left({x}/{\sqrt{T}}\right)\;\;\mbox{for a.e.}\;\;x\in \mathbb{R}^d.
\end{eqnarray}
 Indeed, it follows  by \eqref{equ-basis} and \eqref{sep-tj-thm1-i-pf-1} that
\begin{eqnarray}\label{Wang9122.59}
f_{\sqrt{T}N,n}\big({x}/{\sqrt{T}}\big)=f_{N,n}(x)\;\;\mbox{for each}\;\;x\in \mathbb{R}^d
\end{eqnarray}
and
\begin{eqnarray}\label{Swang9132.60}
u(T,x) = \left(e^{\triangle}v_0\right)\big({x}/{\sqrt{T}}\big)\;\;
\mbox{for each}\;\;x\in \mathbb{R}^d.
\end{eqnarray}
Since by \eqref{sep-tj-thm1-i-pf-3}, we have that
$$
\sum_{n\in\mathbb Z^d} \left(e^\Delta v_0\right)\Big(\frac{n}{\sqrt{T}N}\Big) f_{\sqrt{T}N,n}\in L^2(\mathbb R^d),
$$
 it follows from \eqref{Wang9122.59}
and \eqref{Swang9132.60} that
\begin{eqnarray}\label{SwanG9152.61}
\sum_{n\in\mathbb{Z}^d}u(T,n/N)f_{N,n}\in L^2(\mathbb{R}^d).
\end{eqnarray}
Meanwhile, by \eqref{Swang9132.60}, (\ref{sep-tj-thm1-i-pf-3}), \eqref{Wang9122.59} and \eqref{wang9112.62}, one can directly check  that
\begin{eqnarray}\label{SGW9132.65}
 u(T,x)
 =\sum_{n\in\mathbb Z^d} u(T,{n}/{N}) f_{N,n}(x)
 + R(x)\;\;\mbox{for a.e.}\;\;x\in \mathbb{R}^d.
\end{eqnarray}
Since  $R_1\in L^2(\mathbb R^d)$ (see \eqref{sep-tj-thm1-i-pf-3}), we get from
\eqref{wang9112.62} that
$R\in L^2(\mathbb R^d)$. Thus, \eqref{gwang9122.61}
follows  from \eqref{SGW9132.65}
and \eqref{SwanG9152.61} at once.

 We next estimate $\|R\|$. By \eqref{wang9112.62} and
 (\ref{sep-tj-thm1-i-pf-2}), we see  that
\begin{eqnarray}\label{wang962.60}
 &&\|R\|^2
 = T^{d/2} \| R_1\|^2
 \leq  T^{d/2} C(d)^2   \Big( 1+ (TN^2)^{-d/4} \Big)^2
   \int_{\xi \in Q_{\pi \sqrt{T}N}^c} (1+|\xi|^2)^d  e^{-2|\xi|^2}  |\hat v_0(\xi)|^2 \mathrm d\xi
   \nonumber\\
 &\leq& T^{d/2} C(d)^2   \Big( 1+ (TN^2)^{-d/4} \Big)^2 e^{-2TN^2}
  \int_{\xi \in Q_{\pi \sqrt{T}N}^c} \left(1+|\xi|^2\right)^d  e^{-2(1-\pi^{-2})|\xi|^2}  |\hat v_0(\xi)|^2 \mathrm d\xi.
\end{eqnarray}
Since
$$
C_1(d):=\sup_{\xi\in\mathbb R^d} \left(1+|\xi|^2\right)^d  e^{-2(1-\pi^{-2})|\xi|^2} <\infty,
$$
 it follow from \eqref{wang962.60}  that
\begin{eqnarray}\label{wang9112.60}
\|R(\cdot)\|^2
\leq   C(d)^2 C_1(d)   \Big( 1+ (TN^2)^{-d/4} \Big)^2 e^{-2TN^2} \|u(0,\cdot)\|^2.
\end{eqnarray}
Here, we used the fact that $T^{\frac{d}{2}}\|v_0(\cdot)\|^2=\|u(0,\cdot)\|^2$
which follows from \eqref{sep-tj-thm1-i-pf-1}.

Finally,  from (\ref{gwang9122.61}) and \eqref{wang9112.60}, we are led to
 the conclusion (i) of Theorem \ref{thm-decomp}.

\vskip 5pt
(ii) It is  a direct consequence of  \eqref{equ-coeff-bound-1}
in Lemma \ref{lem-lpq}
  with $r=1/N$.

\vskip 5pt
(iii) Arbitrarily fix $T>0$ and $N>0$. First we suppose that the left hand side of \eqref{wang9} holds. Then we have that
\begin{align}\label{equ-725-1}
\lim_{m\rightarrow +\infty} \big\|u(T,\cdot)-\sum_{|n|\leq m} u(T, {n}/{N})f_{N,n}(\cdot)\big\|=0.
\end{align}
By the Parseval-Plancherel formula, using (i) in Lemma \ref{lem-orth} in Appendix, we get from  \eqref{equ-725-1} that
\begin{align}\label{equ-725-2}
\lim_{m\rightarrow +\infty} \big\|\widehat{u}(T,\cdot)-F_m(\cdot)\big\|=0,
\end{align}
where
$$
F_m(\xi) := \sum_{|n|\leq m} u\left(T, {n}/{N}\right) (2\pi)^{-\frac{d}{2} }N^{-d}e^{-i\frac{n}{N}\cdot\xi}
\chi_{Q_{\pi N}(0)} (\xi),~ \xi\in \mathbb R^d.
$$
 Since
$$
\mbox{supp}\, F_m \subset Q_{\pi N}(0)
\;\;\textmd{for all}\;\;
 m\in \mathbb{N},
$$
it follows from \eqref{equ-725-2} that
$$
\mbox{supp}\,\widehat{u}(T,\cdot)\subset Q_{\pi N}(0).
$$
This, along with the fact:
\begin{eqnarray}\label{GWS9142.71}
\hat u(T,\xi)=e^{-T|\xi|^2}\hat{u}(0,\xi),\;\xi\in\mathbb R^d,
\end{eqnarray}
 yields that
 \begin{eqnarray}\label{GWS9142.72}
  \mbox{supp}\, \hat{u}(0,\cdot)\subset Q_{\pi N}(0).
 \end{eqnarray}
 From \eqref{GWS9142.72} and \eqref{wang1}, we see that  $u(0,\cdot)\in \mathscr{P}_N$, which leads to the right hand side of \eqref{wang9}.

Next, we suppose that the right hand side of \eqref{wang9} is true. Then
 from \eqref{wang1},  we obtain \eqref{GWS9142.72}. This, along with \eqref{GWS9142.71}, yields that
   $\mbox{supp}\, \hat{u}(T,\cdot)\subset Q_{\pi N}(0)$,
   from which, we can use Theorem \ref{lem-shannon} to get the left hand side of \eqref{wang9}.

    \vskip 5pt
    (iv) Arbitrarily fix $T>0$, $N>0$ and  $\varepsilon\in (0,1)$. Arbitrarily take a sequence
$\{\lambda_n\}_{n\in\mathbb Z^d} \subset \mathbb R^d$ so that
\begin{eqnarray}\label{wang10112.26}
\sup_{n\in \mathbb{Z}^d}|\lambda_n-{n}/{N}|\leq {\varepsilon}/{N}.
\end{eqnarray}
According to (i) of Theorem \ref{thm-decomp}, there is $R_1(\cdot) \in L^2(\mathbb R^d)$ and $C_1=C_1(d)>0$, with
\begin{eqnarray}\label{aug-bound-R1}
\|R_1(\cdot)\|\leq C_1(1+({T}N^2)^{-\frac{d}{4}})e^{-TN^2}\|u(0,\cdot)\|,
\end{eqnarray}
 so that
\begin{align}\label{equ-decomp-726-13}
u(T,\cdot) = \sum_{n\in \mathbb{Z}^d} u(T,{n}/{N})f_{N,n}(\cdot) + R_1(\cdot)
\;\;\mbox{in}\;\;
L^2(\mathbb{R}^d).
\end{align}

 We now claim that for some $C_2=C_2(d)>0$,
\begin{eqnarray}\label{wang9122.70}
\Big\| \sum_{n\in \mathbb{Z}^d} \big( u(T,{n}/{N}) - u(T,\lambda_n) \big) f_{N,n}
\Big\|
\leq C_2 \varepsilon
\left( 1+(TN^2)^{-\frac{d}{4}-\frac{1}{2}}e^{-TN^2} \right)
\|u(0,\cdot)\|.
\end{eqnarray}
Indeed, by \eqref{wang10112.26}, we can use  \eqref{wang952.48} in
 Proposition \ref{cor-purb-2} to find  $C_2=C_2(d)>0$ so that
\begin{eqnarray*}\label{equ-decomp-726-13.5}
\big\| \big\{ u(T,\lambda_n)-u\big( T,{n}/{N} \big) \big\}_{n\in \mathbb{Z}^d} \big\|_{l^2(\mathbb{Z}^d)}
\leq C_2 \varepsilon N^{\frac{d}{2}}
\big( 1+(TN^2)^{-\frac{d}{4}-\frac{1}{2}}e^{-TN^2} \big)
\|u(0,\cdot)\|.
\end{eqnarray*}
From  this and the conclusion (iii) in Lemma \ref{lem-orth} in Appendix, where
$a_n$ is taken as $(u(T,\lambda_n)-u\big( T,{n}/{N}))$,
 we are led to  \eqref{wang9122.70}.

Since $\sum_{n\in\mathbb{Z}^d}u(T,n/N)f_{N,n}\in L^2(\mathbb{R}^d)$, it follows from
\eqref{wang9122.70} that
\begin{eqnarray}\label{wang9122.71}
\sum_{n\in\mathbb{Z}^d}u(T,\lambda_n)f_{N,n}\in L^2(\mathbb{R}^d)
\;\;\mbox{and}\;\;\sum_{n\in \mathbb{Z}^d} \big( u(T,{n}/{N}) - u(T,\lambda_n) \big) f_{N,n}\in L^2(\mathbb{R}^d).
\end{eqnarray}
From \eqref{equ-decomp-726-13} and \eqref{wang9122.71}, we find that
\begin{align}\label{equ-decomp-726-14}
u(T,\cdot) = \sum_{n\in \mathbb{Z}^d} u(T,\lambda_n)f_{N,n}(\cdot) + \widetilde{R}(\cdot)\;\;\mbox{in}\;\;
L^2(\mathbb{R}^d),
\end{align}
where
\begin{align}\label{equ-decomp-726-15}
\widetilde{R}(\cdot) := R_1(\cdot) + \sum_{n\in \mathbb{Z}^d} \big( u(T,{n}/{N})-u\big(T,\lambda_n \big) \big)f_{N,n}(\cdot)\in L^2(\mathbb{R}^d).
\end{align}
Meanwhile, by  (\ref{equ-decomp-726-15}),  (\ref{aug-bound-R1}) and \eqref{wang9122.70}, we obtain that
\begin{eqnarray*}\label{equ-decomp-726-17}
\|\widetilde{R}\| \leq \max\{C_1, C_2\}
\Big(
\varepsilon +
\Big( 1+(TN^2)^{-\frac{1}{2}} \Big)
(TN^2)^{-\frac{d}{4}}e^{-TN^2}
\Big)
\|u(0,\cdot)\|,
\end{eqnarray*}
which, along with (\ref{equ-decomp-726-14}), leads to \eqref{wang11} and \eqref{wang10}.

\vskip 5pt
In summary, we end the proof of
 Theorem \ref{thm-decomp}.

\end{proof}

\section{Weak asymptotic observability identity}

We first  shows that the asymptotic observability identity \eqref{equ-decomp-725} is not true when the sum on the right hand side of (\ref{equ-decomp-725}) is taken over finite lattice  points.

\begin{proposition}\label{prop-cout-ex}
For any $T>0$, $N>0$ and any $\mathcal{G}(\cdot)\in C(\mathbb R^+,\mathbb R^+)$,
it holds that
\begin{eqnarray}\label{wang9263.1}
 \sup \Big\|
  u(T,\cdot) - \sum_{n\in \mathbb{Z}^d,|n|\leq \mathcal{G}(N)}
  u\big(T,{n}/{N}\big)f_{N,n}(\cdot) \Big\|
 \geq  \frac{1}{2}(T+1)^{-\frac{d}{4}},
\end{eqnarray}
where the supremum is taken over all solution $u$ to \eqref{equ-heat}
with $\|u(0,\cdot)\|\leq 1$.
\end{proposition}

\begin{proof}
Arbitrarily fix $T>0$,  $N>0$ and $\mathcal{G}(\cdot)\in C(\mathbb R^+,\mathbb R^+)$. We define
\begin{eqnarray}\label{equ-9-15-3}
\begin{array}{lll}
L_N&:=&\mathcal{G}(N)/N+ \sqrt{2(T+1) \Big[ \big( 2(\mathcal{G}(N)+1) \big)^d + 2N^{-d/2} + \ln 4^{1+d/4} \Big] };
\\
u_{N}(t,x) &:=& (4\pi (t+1))^{-d/2} \exp\left\{-\frac{|x-(L_N,0,\cdots,0)|^2}{4(t+1)}\right\},~ x\in\mathbb{R}^d, ~t\geq 0.
\end{array}
\end{eqnarray}
Then one can easily check that  $u_{N}$ solves the heat equation \eqref{equ-heat}
and satisfies that
\begin{eqnarray}\label{equ-9-15-3-00}
u_{N}(0,x) = (4\pi )^{-d/2}\exp\left\{-\frac{|x-(L_N,0,\cdots,0)|^2}{4}\right\}, ~ x\in\mathbb{R}^d.
\end{eqnarray}
We now claim that
\begin{align}\label{equ-9-15-4.5}
\Big\|\sum_{ n\in \mathbb{Z}^d,|n|\leq \mathcal{G}(N)} u_{N}\big(T,{n}/{N}\big)f_{N,n}(\cdot)\Big\| \leq
\frac{1}{2} (8\pi)^{-\frac{d}{4}}(T+1)^{-\frac{d}{4}}.
\end{align}
Indeed, it follows by   (ii) of Lemma \ref{lem-orth} that
\begin{eqnarray*}
  \Big\|\sum_{|n|\leq \mathcal{G}(N), n\in \mathbb{Z}^d} u_{N}\big(T,{n}/{N}\big)f_{N,n}\Big\|
 = N^{-d/2}\Big(\sum_{n\in \mathbb{Z}^d,|n|\leq \mathcal{G}(N)} u_{N}^2\big(T,{n}/{N}\big)\Big)^{1/2}.
\end{eqnarray*}
This, together with the second equality in (\ref{equ-9-15-3}), implies that
\begin{eqnarray}\label{equ-9-15-5}
& & \Big\|\sum_{|n|\leq \mathcal{G}(N), n\in \mathbb{Z}^d} u_{N}\big(T,{n}/{N}\big)f_{N,n}\Big\|
\nonumber\\
&=&  N^{-d/2}(4\pi (T+1))^{-d/4}
   \left( \sum_{n\in \mathbb{Z}^d,|n|\leq \mathcal{G}(N)} \exp\left\{-\frac{|\frac{n}{N}-(L_N,0,\cdots,0)|^2}{2(T+1)}\right\}
\right)^{1/2}.
\end{eqnarray}
Meanwhile, it follows from the first equality in (\ref{equ-9-15-3}) that
 \begin{eqnarray}\label{wang9263.6}
   \sup_{n\in \mathbb{Z}^d,|n|\leq \mathcal{G}(N)} |{n}/{N}-(L_N,0,\cdots,0) |^2
   &=& |L_N- \mathcal{G}(N)/N|^2
   \nonumber\\
   &=&  2(T+1) \Big[ \big( 2(\mathcal{G}(N)+1) \big)^d + 2N^{-d/2} + \ln 4^{1+d/4} \Big].
 \end{eqnarray}
 Notice that  the set $\{n\in \mathbb Z^d~:~|n|\leq \mathcal{G}(N)\}$ has at most $\big( 2(\mathcal{G}(N)+1) \big)^d$ elements. This, along with (\ref{equ-9-15-5}) and \eqref{wang9263.6}, yields that
 \begin{eqnarray*}
 \Big\|\sum_{|n|\leq \mathcal{G}(N), n\in \mathbb{Z}^d} u_{N}\big(T,{n}/{N}\big)f_{N,n}\Big\|
&\leq&  N^{-d/2}(4\pi (T+1))^{-d/4}
\nonumber\\
& & \times
   \left(  \big( 2(\mathcal{G}(N)+1) \big)^d e^{ -\big(-2(\mathcal{G}(N)+1)\big)^d}  e^{-2N^{-d/2}}  4^{-1-d/4}
\right)^{1/2}
\nonumber\\
&\leq& \frac{1}{2} (8\pi)^{-\frac{d}{4}}(T+1)^{-\frac{d}{4}},
\end{eqnarray*}
which  leads to  (\ref{equ-9-15-4.5}).

Finally, after some computations, we see from (\ref{equ-9-15-3}) and (\ref{equ-9-15-3-00}) that
\begin{eqnarray}\label{equ-9-15-5-1}
\|u_{N}(T,\cdot)\|= (8\pi)^{-\frac{d}{4}}(T+1)^{-\frac{d}{4}}
\;\;\mbox{and}\;\;
\|u_{N}(0,\cdot)\|=  (8\pi)^{-\frac{d}{4}}.
\end{eqnarray}
 From   (\ref{equ-9-15-4.5}) and the first equality in  (\ref{equ-9-15-5-1}), we get that for each $N>0$,
 \begin{eqnarray*}
   \Big\| u_N(T,\cdot) - \sum_{ n\in \mathbb{Z}^d,|n|\leq \mathcal{G}(N)} u_{N}\big(T,{n}/{N}\big)f_{N,n}(\cdot)\Big\|
   &\geq& \| u_N(T,\cdot)\|  -  \Big\| \sum_{ n\in \mathbb{Z}^d,|n|\leq \mathcal{G}(N)} u_{N}\big(T,{n}/{N}\big)f_{N,n}(\cdot)\Big\|
   \nonumber\\
   &\geq& \frac{1}{2} (8\pi)^{-\frac{d}{4}}(T+1)^{-\frac{d}{4}},
 \end{eqnarray*}
 which, along with the second equality in  (\ref{equ-9-15-5-1}),
 leads to \eqref{wang9263.1}. This ends the proof of Proposition \ref{prop-cout-ex}.
\end{proof}

Next we will introduce a weak asymptotic observability identity with finite many observation
lattice points. This  identity holds only for some solutions to
Equation \eqref{equ-heat}.
The main result of this section is as follows:

\begin{theorem}\label{thm-asym-ob-balls}
With the notations in  (\ref{equ-basis}) and (\ref{wang1}),
there is a positive constant $C=C(d)$, depending only on $d$, so that
any solution  $u$ to \eqref{equ-heat}, with
$\int_{\mathbb R^d} (1+|x|)^{2} |u(0,x)|^2 \mathrm dx <\infty$,
has  the following properties:
 Given $T>0$, $N>0$ and $r\geq 1$, there is   $R(\cdot)\in L^2(\mathbb R^d)$,
  with
 \begin{eqnarray*}
 \|R(\cdot)\| &\leq&  C\left(1+\left(TN^2\right)^{-\frac{d}{4}} \right) \Big( e^{-TN^2}
 +
      (1+T^{ \frac{d}{2} }) (1+T^{- \frac{1}{2} }) r^{-1}   \Big)
      \nonumber\\
      & & \times   \Big( \int_{\mathbb R^d} (1+|x|)^{2} |u(0,x)|^2 \mathrm dx \Big)^{1/2},
 \end{eqnarray*}
       so that
 \begin{eqnarray*}
 u(T,\cdot) = \sum_{ n\in \mathbb{Z}^d,|n/N|<r } u\big(T,{n}/{N}\big)f_{N,n}(\cdot)+R(\cdot) \;\;\mbox{in}\;\; L^2(\mathbb R^d).
 \end{eqnarray*}
\end{theorem}
To prove Theorem \ref{thm-asym-ob-balls}, we first show the next lemma.
Throughout the rest of this section, we write $C_\alpha^\beta:=C_{\alpha_1}^{\beta_1}\cdots C_{\alpha_d}^{\beta_d}$ and
 $\beta\leq \alpha$, when $\alpha=(\alpha_1,\dots,\alpha_d), \beta=(\beta_1,\dots,\beta_d)\in \mathbb N^d$ satisfy that $\beta_j\leq \alpha_j$ for all $j\in\{1,\dots,d\}$.

\begin{lemma}\label{lem-heat-decreasing-forward}
Let $T>0$, $\alpha\in \mathbb N^d$ and $k\in\mathbb N$. Then any solution  $u$ to   \eqref{equ-heat},
with $\int_{\mathbb R^d} (1+|x|)^{2k} |u(0,x)|^2 \mathrm dx<\infty$,
 satisfies that
\begin{eqnarray}\label{decresing-estimate-for-heat-sep-0}
 \int_{\mathbb R^d} (1+|x|)^{2k} | D_x^\alpha u(T,x)|^2 \mathrm dx
 \leq (2d)^{k+1} \Big( 6^k (|\alpha|+k)! \Big)^2   (1+T)^{k} T^{- |\alpha|  } \int_{\mathbb R^d} (1+|x|)^{2k} |u(0,x)|^2 \mathrm dx.
\end{eqnarray}

\end{lemma}

\begin{proof}
 Arbitrarily fix $T>0$, $\alpha\in \mathbb N^d$ and $k\in\mathbb N$. Then arbitrarily fix a solution $u$ to \eqref{equ-heat} so that $\int_{\mathbb R^d} (1+|x|)^{2k} |u(0,x)|^2 \mathrm dx<\infty$.
 We claim that for each $\beta\in\mathbb N^d$,
 \begin{eqnarray}\label{decresing-estimate-for-heat-sep-3}
 \int_{\mathbb R^d_x}  | x^\beta D_x^\alpha u(T,x)|^2 \mathrm dx
 \leq  2^{|\beta|+1} \Big( 3^{|\beta|} (|\alpha|+|\beta|)! \Big)^2 (1+T)^{|\beta|}  T^{-|\alpha|}
 \int_{\mathbb R^d_x}    (1+|x|)^{2|\beta|} | u(0,x) |^2 \mathrm dx.
\end{eqnarray}
To this end, we arbitrarily fix $
\beta=(\beta_1,\dots,\beta_d)\in \mathbb N^d$.
 By the Parseval-Plancherel formula, we obtain that
\begin{eqnarray*}
  \int_{\mathbb R^d_x}  | x^\beta D_x^\alpha u(T,x)|^2 \mathrm dx
  &=& \int_{\mathbb R^d_\xi}  | D_\xi^\beta \big( \xi^\alpha \widehat u(T,\xi) \big)|^2 \mathrm d\xi
  = \int_{\mathbb R^d_\xi}  | D_\xi^\beta \big( \xi^\alpha e^{-T|\xi|^2} \widehat u(0,\xi) \big)|^2 \mathrm d\xi
  \nonumber\\
  &=& T^{|\beta|-|\alpha|} T^{\frac{d}{2}} \int_{\mathbb R^d_\eta}  | D_\eta^\beta \big( \eta^\alpha e^{-|\eta|^2} \widehat u(0,\eta/\sqrt{T}) \big)|^2 \mathrm d\eta
  \nonumber\\
  &=& T^{|\beta|-|\alpha|} T^{\frac{d}{2}} \int_{\mathbb R^d_\eta}  \Big| \sum_{\gamma\in \mathbb N^d,\gamma \leq \beta } C_\beta^\gamma D_\eta^{\beta-\gamma} \big( \eta^\alpha e^{-|\eta|^2} \big) D_\eta^\gamma \widehat u(0,\eta/\sqrt{T})  \Big|^2 \mathrm d\eta.
\end{eqnarray*}
This yields that
\begin{eqnarray}\label{decresing-estimate-for-heat-sep-1}
 \int_{\mathbb R^d_x}  | x^\beta D_x^\alpha u(T,x)|^2 \mathrm dx
  &\leq&  T^{|\beta|-|\alpha|+\frac{d}{2}} \times
  \\
  & &
    2\sum_{ \gamma\in \mathbb N^d,\gamma \leq \beta }  \Big[  \Big(
   \sup_{\eta\in \mathbb R^d}  \big| D_\eta^{\beta-\gamma} \big( \eta^\alpha e^{-|\eta|^2} \big) \big| \Big)^2
  \int_{\mathbb R^d_\eta}  | C_\beta^\gamma D_\eta^\gamma  \widehat u(0,\eta/\sqrt{T}) |^2 \mathrm d\eta
  \Big].
  \nonumber
\end{eqnarray}
Meanwhile, by using induction, one can easily obtain  that
\begin{eqnarray*}
  \sup_{s\in \mathbb R} \big| D_s^{m} \big( s^n e^{-s^2} \big) \big|
  \leq (m+n)! 3^m
  \;\;\mbox{for each}\;\;
  m,n\in\mathbb N,
\end{eqnarray*}
from which, it follows that for each $\gamma=(\gamma_1,\dots,\gamma_d)\in \mathbb N^d$ with $\gamma\leq \beta$,
\begin{eqnarray*}
 \sup_{\eta\in \mathbb R^d} \big| D_\eta^{\beta-\gamma} \big( \eta^\alpha e^{-|\eta|^2} \big) \big|
 &=& \prod_{j=1}^d \sup_{\eta_j\in \mathbb R} \big| D_{\eta_j}^{\beta_j-\gamma_j} \big( \eta_j^{\alpha_j} e^{-\eta_j^2} \big) \big|
 \nonumber\\
 &\leq& \prod_{j=1}^d \Big[ (\alpha_j+\beta_j-\gamma_j)! 3^{\beta_j-\gamma_j} \Big]
 \leq (|\alpha|+|\beta|)! 3^{|\beta|}.
\end{eqnarray*}
This, together with (\ref{decresing-estimate-for-heat-sep-1}), indicates that
\begin{eqnarray*}
 \int_{\mathbb R^d_x}  | x^\beta D_x^\alpha u(T,x)|^2 \mathrm dx
  &\leq&  2\Big( 3^{|\beta|} (|\alpha|+|\beta|)! \Big)^2  T^{|\beta|-|\alpha|} \times
  \\
  & &
    \sum_{ \gamma\in \mathbb N^d,\gamma \leq \beta }
  \int_{\mathbb R^d_\eta} T^{-|\gamma|} | C_\beta^\gamma  D_\eta^\gamma  \widehat u(0,\eta) |^2 \mathrm d\eta.
  \nonumber
\end{eqnarray*}
Making use of the Parseval-Plancherel formula to the above, we deduce that
\begin{eqnarray}\label{decresing-estimate-for-heat-sep-2}
 & & \int_{\mathbb R^d_x}  | x^\beta D_x^\alpha u(T,x)|^2 \mathrm dx
 \nonumber\\
  &\leq&   2\Big( 3^{|\beta|} (|\alpha|+|\beta|)! \Big)^2 (1+T)^{|\beta|}  T^{-|\alpha|}
    \sum_{ \gamma\in \mathbb N^d,\gamma \leq \beta }
  \int_{\mathbb R^d_x}    | C_\beta^\gamma x^\gamma  u(0,x) |^2 \mathrm dx.
\end{eqnarray}
Since
\begin{eqnarray*}
   \sum_{ \gamma\in \mathbb N^d,\gamma \leq \beta }
   | C_\beta^\gamma x^\gamma|^2
   &\leq&   2^{|\beta|}   \sum_{ \gamma\in \mathbb N^d,\gamma \leq \beta }
    C_\beta^\gamma x_1^{2\gamma_1}\cdots x_d^{2\gamma_d}
   \nonumber\\
   &=& 2^{|\beta|} \prod_{j=1}^d \sum_{0\leq \gamma_j \leq \beta_j} C_{\beta_j}^{\gamma_j} x_j^{2\gamma_j}
   = 2^{|\beta|} \prod_{j=1}^d  (1+x_j^2)^{\gamma_j}
   \nonumber\\
   &\leq& 2^{|\beta|}(1+x_1^2+\cdots+x_d^2)^{|\beta|}
   \leq 2^{|\beta|}(1+|x|)^{2|\beta|}\;\;\mbox{for each}\;\;x=(x_1,\dots,x_d)\in\mathbb R^d,
\end{eqnarray*}
we get (\ref{decresing-estimate-for-heat-sep-3}) from  (\ref{decresing-estimate-for-heat-sep-2}) immediately.

Finally, since
\begin{eqnarray*}
 (1+|x|)^{2k}
 &\leq& \max_{1\leq j \leq d} \Big(1+\sqrt{d}|x_l| \Big)^{2k}
 \leq d^k\sum_{l=1}^d (1+|x_l|)^{2k}
 \nonumber\\
 &\leq& (2d)^k \sum_{l=1}^d (1+x_l^2)^{k}
 = (2d)^k \sum_{l=1}^d \sum_{j=0}^k C_k^j x_l^{2j}\;\;\mbox{for each}\;\;x=(x_1,\dots,x_d)\in\mathbb R^d,
\end{eqnarray*}
it follows from  (\ref{decresing-estimate-for-heat-sep-3}) that
\begin{eqnarray*}
& & \int_{\mathbb R^d} (1+|x|)^{2k} | D_x^\alpha u(T,x)|^2 \mathrm dx
\leq (2d)^k \sum_{l=1}^d \sum_{j=0}^k C_k^j  \int_{\mathbb R^d}    | x_l^{j} D_x^\alpha u(T,x)|^2 \mathrm dx
\nonumber\\
&\leq& (2d)^k 2^{k+1} \sum_{l=1}^d \sum_{j=0}^k C_k^j  \Big( 3^j (|\alpha|+j)! \Big)^2 (1+T)^{j}  T^{-|\alpha|}
 \int_{\mathbb R^d_x}    (1+|x|)^{2j} | u(0,x) |^2 \mathrm dx
 \nonumber\\
 &\leq& (2d)^k  2^{k+1} d 2^k \Big( 3^k (|\alpha|+k)! \Big)^2 (1+T)^{k}  T^{-|\alpha|}
 \int_{\mathbb R^d_x}    (1+|x|)^{2k} | u(0,x) |^2 \mathrm dx,
\end{eqnarray*}
which leads to the desired result. This ends the proof of Lemma \ref{lem-heat-decreasing-forward}.
\end{proof}

We are in the position to prove Theorem \ref{thm-asym-ob-balls}.

\begin{proof}[Proof of Theorem \ref{thm-asym-ob-balls}]
 We claim the following statement:
\begin{itemize}
  \item[($\mathcal P$)] There is $C(d)>0$ so that for
 any $T>0$, $N>0$, $k\in\mathbb N$, $r\geq 1$ and any solution
    $u$ to \eqref{equ-heat}, with
$\int_{\mathbb R^d} (1+|x|)^{2k} |u(0,x)|^2 \mathrm dx <\infty$,
    there is $R(\cdot)\in L^2(\mathbb R^d)$,
  with the estimate:
 \begin{eqnarray*}
 \|R(\cdot)\| &\leq&  C(d)\left(1+\left(TN^2\right)^{-\frac{d}{4}} \right) \Big( e^{-TN^2}
 +
    d^{\frac{k}{2}} 12^k (d+k)!  (1+r^{-d}T^{\frac{d}{2}}) (1+T^{-\frac{k}{2}} ) (1+r)^{-k}   \Big)
      \nonumber\\
      & & \times   \Big( \int_{\mathbb R^d} (1+|x|)^{2k} |u(0,x)|^2 \mathrm dx \Big)^{ \frac{1}{2} },
 \end{eqnarray*}
       so that
 \begin{eqnarray*}
 u(T,\cdot) = \sum_{ n\in \mathbb{Z}^d,|n/N|<r } u\big(T,{n}/{N}\big)f_{N,n}(\cdot)+R(\cdot) \;\;\mbox{in}\;\; L^2(\mathbb R^d).
 \end{eqnarray*}
\end{itemize}
It is clear  that the desired conclusion in Theorem \ref{thm-asym-ob-balls}
is exactly the conclusion ($\mathcal P$) where $k=1$.

To prove
 ($\mathcal P$), we arbitrarily fix required $T$, $N$, $k$, $r$
 and  $u$ as required.  Set
\begin{eqnarray}\label{sep-heat-finite-estimate-3}
 v_0(x):=u\left(0,\sqrt{T}x\right)\;\;\mbox{for a.e.}\;\;x\in\mathbb R^d.
\end{eqnarray}
It is clear that
\begin{eqnarray}\label{sep-heat-finite-estimate-4}
u(T,x) = \left(e^{\triangle}v_0\right)\big({x}/{\sqrt{T}}\big)\;\;
\mbox{for each}\;\;x\in \mathbb{R}^d.
\end{eqnarray}

\vskip 5pt
\noindent
\textit{Step 1. We prove that for some $C^\prime(d)>0$,
\begin{eqnarray}\label{sep-heat-finite-estimate-5}
 \Big\|  \sum_{n\in\mathbb Z^d,|n/N|>r} (e^{\triangle}v_0)\Big( \frac{n}{N\sqrt{T}} \Big) f_{N\sqrt{T},n}  \Big\|
 &\leq&  C^\prime(d)
    d^{k/2} 12^k (d+k)! (1+r^{-d}T^{d/2}) (1+r)^{-k}  (1+T^{-k/2})
      \nonumber\\
      & & \times  T^{-d/4}  \Big( \int_{\mathbb R^d} (1+|x|)^{2k} |u(0,x)|^2 \mathrm dx \Big)^{1/2}.
\end{eqnarray}}
  By (iii) of Lemma \ref{lem-orth}, it follows that
\begin{eqnarray}\label{sep-heat-finite-estimate-6}
 \Big\|\sum_{n\in\mathbb Z^d,|n/N|>r} (e^{\triangle}v_0) \Big( \frac{n}{N\sqrt{T}} \Big) f_{N,n}\Big\|
 = (TN^2)^{-d/4} \Big(
\sum_{n\in\mathbb Z^d,|n/N|>r} \Big|(e^{\triangle}v_0)\Big( \frac{n}{N\sqrt{T}} \Big) \Big|^2 \Big)^{1/2}.
\end{eqnarray}
Take $\rho\in C_0^\infty(\mathbb R^d)$ so that
\begin{itemize}
  \item[(i)] $\rho(x)=1$, when $|x|>\frac{r}{\sqrt{T}}$, \;\;\;  $\rho(x)=0$,
      when $|x|\leq \frac{r}{2\sqrt{T}}$;
  \item[(ii)] there exists
  $C_1(d)>0$ so that  for each $\alpha\in \mathbb N^d$ with $|\alpha|\leq d$,
  \begin{eqnarray*}
  \sup_{x\in\mathbb R^d} |D^\alpha \rho(x)| \leq  C_1(d) \Big(\frac{r}{\sqrt{T}} \Big)^{-|\alpha|}.
  \end{eqnarray*}
\end{itemize}
Then we can apply Proposition \ref{lem-smooth-l2-r} (where $s=d$, $f(\cdot)=\rho(\cdot) u(T,\cdot)$ and $r=1/(N\sqrt{T})$), as well as (\ref{sep-heat-finite-estimate-6}),  to find  $C(d)>0$ so that
\begin{eqnarray}\label{sep-heat-finite-estimate-2}
\Big\|\sum_{n\in\mathbb Z^d,|n/N|>r} (e^{\triangle}v_0) \Big( \frac{n}{N\sqrt{T}} \Big) f_{N,n}\Big\|
&=& \Big(\sum_{n\in\mathbb Z^d,|n/N|>r} |\rho(n/N) (e^{\triangle}v_0) \Big( \frac{n}{N\sqrt{T}} \Big)|^2 \Big)^{1/2}
\nonumber\\
&\leq&  C(d) (1+(TN^2)^{-d/4}) \|\rho e^{\triangle}v_0 \|_{H^{d}(\mathbb R^d)}.
\end{eqnarray}

 We now estimate the term $\|\rho e^{\triangle}v_0 \|_{H^{d}(\mathbb R^d)}$.
 First, according to  the properties of $\rho$, there is  $C_2(d)>0$ so that
\begin{eqnarray*}
 & &  \|\rho e^{\triangle}v_0 \|_{H^d(\mathbb R^d)}
  \leq C_2(d) \Big( \|\rho e^{\triangle}v_0 \|
  +  \sum_{\alpha\in \mathbb N^d,|\alpha|=d} \|D_x^\alpha(\rho e^{\triangle}v_0)\|\Big)
  \nonumber\\
  &\leq&  C_2(d) \Big( \|\rho e^{\triangle}v_0 \|
  +  \sum_{\alpha\in \mathbb N^d,|\alpha|=d}
  \sum_{\beta\in\mathbb N^d,\beta\leq \alpha} C_\alpha^\beta \|D^{\alpha-\beta}_x \rho  D^{\beta}_x e^{\triangle}v_0\| \Big)
  \nonumber\\
  &\leq& C_1(d)C_2(d)  \Big( \|e^{\triangle}v_0 \|_{L^2(B^c_{r/(2\sqrt{T})}(0))}
  +  \sum_{\alpha\in \mathbb N^d,|\alpha|=d}
  \sum_{\beta\in\mathbb N^d,\beta\leq \alpha} C_\alpha^\beta \Big( \frac{r}{2\sqrt{T}} \Big)^{-|\alpha|+|\beta|} \|D^{\beta}_x e^{\triangle}v_0\|_{L^2(B^c_{r/(2\sqrt{T})}(0))} \Big),
\end{eqnarray*}
from which, it follows  that for some $C_3(d)>0$,
\begin{eqnarray}\label{sep-heat-finite-estimate-1}
  \|\rho e^{\triangle}v_0 \|_{H^d(\mathbb R^d)}
  \leq C_3(d) \Big( 1+ \Big(\frac{r}{\sqrt{T}}\Big)^{-d} \Big)
  \sum_{\beta\in\mathbb N^d,|\beta|\leq d}  \|D^{\beta}_x e^{\triangle}v_0\|_{L^2(B^c_{r/(2\sqrt{T})}(0))}.
\end{eqnarray}
Second, we can use Lemma \ref{lem-heat-decreasing-forward} (where $T=1$ and $u(0,\cdot)=v_0(\cdot)$) to see that for each $\beta\in\mathbb N^d$ with $|\beta|\leq d$,
\begin{eqnarray}\label{gwang3.181012}
 \|D^{\beta}_x e^{\triangle}v_0\|_{L^2(B^c_{r/(2\sqrt{T})}(0))}^2
 &=& \int_{|x|\geq r/(2\sqrt{T})}  | D^{\beta}_x e^{\triangle}v_0|^2 \mathrm dx
 \nonumber\\
 &\leq&
 (1+r/(2\sqrt{T}))^{-2k} \int_{\mathbb R^d} (1+|x|)^{2k} | D^{\beta}_x e^{\triangle}v_0(x)|^2 \mathrm dx\\
 &\leq& (2d)^{k+1} \Big( 6^k (d+k)! \Big)^2 (1+r/(2\sqrt{T}))^{-2k}  2^{k}    \int_{\mathbb R^d} (1+|x|)^{2k} |v_0(x)|^2 \mathrm dx.\nonumber
\end{eqnarray}
From (\ref{sep-heat-finite-estimate-1})
and \eqref{gwang3.181012}, we see that  for some $C_4(d)>0$,
\begin{eqnarray*}
  \|\rho e^{\Delta} v_0 \|_{H^d(\mathbb R^d)}
  \leq C_4(d)
    d^{k/2} 12^k (d+k)! (1+r^{-d}T^{d/2}) (1+rT^{-1/2})^{-k}     \Big( \int_{\mathbb R^d} (1+|x|)^{2k} |v_0(x)|^2 \mathrm dx \Big)^{1/2}.
\end{eqnarray*}
This, together with  (\ref{sep-heat-finite-estimate-2}) and (\ref{sep-heat-finite-estimate-3}), leads to (\ref{sep-heat-finite-estimate-5}).

\vskip 5pt
\noindent
\textit{Step 2. We  prove the conclusion ($\mathcal P$).}

\noindent It follows from (\ref{sep-heat-finite-estimate-4}), (\ref{Wang9122.59}) and  (\ref{sep-heat-finite-estimate-5}) that
\begin{eqnarray*}\label{sep-heat-finite-estimate-10}
 \Big\|   \sum_{n\in\mathbb Z^d,|n/N|>r} u(T, n/N) f_{N,n}  \Big\|
 &\leq&  C^\prime(d)
    d^{k/2} 12^k (d+k)! (1+r^{-d}T^{d/2}) (1+r)^{-k}  (1+T^{-k/2})
      \nonumber\\
      & & \times   \Big( \int_{\mathbb R^d} (1+|x|)^{2k} |u(0,x)|^2 \mathrm dx \Big)^{1/2}.
\end{eqnarray*}
This, along with  Theorem \ref{thm-decomp}, leads to  ($\mathcal P$). Thus we  end the proof of Theorem \ref{thm-asym-ob-balls}.
\end{proof}

\section{Application to  controllability}

 In this  section, we will  apply Theorem \ref{thm-decomp}, as well as Theorem \ref{thm-asym-ob-balls},
  to build up some kind of feedback null approximate controllability for some impulsively controlled heat equations.
  We start with introducing the  controlled system for the application of Theorem \ref{thm-decomp}.
  Throughout this section, we arbitrarily fix $T>\tau>0$.
  For each $N>0$, we define a control operator $B_N$ from $l^2(\mathbb Z^d)$ to $\mathcal{D}^\prime(\mathbb R^d)$ by setting
 \begin{eqnarray}\label{def-control-operator-N}
  B_Nv := \sum_{n\in \mathbb Z^d} v_n \delta_{N,n}
  \;\;\mbox{for each}\;\;
  v=\{v_n\}_{n\in \mathbb Z^d} \in l^2(\mathbb Z^d),
 \end{eqnarray}
 where $\delta_{N,n}(x):=\delta(x-\frac{n}{N})$, $x\in \mathbb{R}^d$.
   Consider the following control system:
 \begin{eqnarray}\label{heat-impulsive-control-system-N}
  \left\{
   \begin{array}{lll}
     \partial_t y  - \Delta y =0  &\mbox{in}  & \big((0,T)\setminus\{\tau\} \big) \times \mathbb R^d,\\
     y|_{t=\tau}=y|_{t=\tau-} + B_N v   &\mbox{in}  &\mathbb R^d,\\
     y|_{t=0}=y_0  &\mbox{in}  &\mathbb R^d,
   \end{array}
  \right.
 \end{eqnarray}
 where $N>0$, $y_0\in L^2(\mathbb R^d)$, $v\in l^2(\mathbb Z^d)$ is a control
 and $y|_{t=\tau-}$ denotes the left limit of  $y$ (which is treated as a function  from $\mathbb{R}^+$ to $\mathbb{R}^d$) at time $\tau$.
   Notice that in the system (\ref{heat-impulsive-control-system-N}), controls are added impulsively in both time and space.
   It will be seen  in Lemma \ref{lemma973.1} that for each $y_0\in L^2(\mathbb R^d)$, each $v\in l^2(\mathbb Z^d)$ and each $N>0$,  the equation (\ref{heat-impulsive-control-system-N}) has a unique solution in some space.
        Write $y(\cdot; N,y_0,v)$  for this  solution.

 \begin{lemma}\label{lemma973.1}
  Let  $N>0$ and $s>\frac{d}{2}$. The following two conclusions are true:

  (i) The control operator $B_N$ is linear and bounded from $ l^2(\mathbb Z^d)$ to $H^{-s}(\mathbb R^d)$.

  (ii) If $y_0\in L^2(\mathbb R^d)$ and $v\in l^2(\mathbb Z^d)$, then
  the unique solution to
  \eqref{heat-impulsive-control-system-N} satisfies
  $$
  y(\cdot; N,y_0,v)\in C\big([0,\tau)\cup(\tau,T];L^2(\mathbb R^d)\big)
  \;\;\mbox{and}\;\;y(\cdot; N,y_0,v)|_{[\tau,T]}\in C\big([\tau,T]; H^{-s}(\mathbb R^d)\big).
  $$
  Furthermore,  $y(t; N,y_0,v)|_{t=\tau-}$  exists in $L^2(\mathbb R^d)$.
 \end{lemma}

 \begin{proof}
  Arbitrarily fix  $N>0$ and $s>\frac{d}{2}$. The conclusion (i)-(ii) will be proved one by one.

  \vskip 5pt
  (i) Arbitrarily take $v=\{v_n\}_{n\in \mathbb Z^d} \in l^2(\mathbb Z^d)$.
  Three facts are given in order. Fact One: We have that
  \begin{eqnarray}\label{wang973.3}
  \| B_N v \|_{H^{-s}(\mathbb R^d)} &=& \sup_{\psi\in C_0^\infty(\mathbb R^d),\|\psi\|_{H^s(\mathbb R^d)}\leq 1}
     \langle B_N v, \psi \rangle_{H^{-s}(\mathbb R^d),H^s(\mathbb R^d)}
     \nonumber\\
     &=& \sup_{\psi\in C_0^\infty(\mathbb R^d),\|\psi\|_{H^s(\mathbb R^d)} \leq 1} \langle B_N v, \psi \rangle_{\mathcal{D}^\prime(\mathbb R^d),C_0^\infty(\mathbb R^d)}.
       \end{eqnarray}
  Fact Two: It follows from  (\ref{def-control-operator-N}) that for each $\psi\in C_0^\infty(\mathbb R^d)$,
  \begin{eqnarray}\label{wang973.4}
   \langle B_N v, \psi \rangle_{ \mathcal{D}^\prime(\mathbb R^d), C_0^\infty(\mathbb R^d) }
   = \sum_{n\in \mathbb Z^d} v_n \psi({n}/{N})
   \leq \Big( \sum_{ n\in\mathbb Z^d } v_n^2 \Big)^{\frac{1}{2}}
   \Big( \sum_{ n\in\mathbb Z^d } \psi\big({n}/{N} \big)^2 \Big)^{\frac{1}{2}}.
  \end{eqnarray}
  Fact Three: For each $\psi\in C_0^\infty(\mathbb R^d)$, we can apply Proposition \ref{lem-smooth-l2-r}, where $r=\frac{1}{N}$ and $f=\psi$, to
  find $C(s,d)>0$ so that
  \begin{eqnarray}\label{wang973.5}
  \Big( \sum_{ n\in\mathbb Z^d } \psi\big({n}/{N} \big)^2 \Big)^{\frac{1}{2}}
  \leq C(s,d)\left(1+N^{\frac{d}{2}}\right)\|\psi\|_{H^s(\mathbb{R}^d)}.
  \end{eqnarray}

  Now, from \eqref{wang973.3}, \eqref{wang973.4} and \eqref{wang973.5}, it follows that
    \begin{eqnarray*}
    \| B_N v \|_{H^{-s}(\mathbb R^d)}
     \leq  C(s,d) \left(1+N^{\frac{d}{2}}\right) \|v\|_{l^2(\mathbb Z^d)},
  \end{eqnarray*}
  which  leads to the conclusion (i).

  \vskip 5pt
  (ii) Arbitrarily fix $y_0\in L^2(\mathbb R^d)$ and $v\in l^2(\mathbb Z^d)$. It is clear that the unique solution
  to \eqref{heat-impulsive-control-system-N} can be expressed as
  \begin{eqnarray}\label{wang973.6}
    y(t; N,y_0,v)
    =\left\{
     \begin{array}{ll}
     e^{t\Delta}y_0,~&0\leq t<\tau,\\
     e^{t\Delta}y_0 + (e^{(t-\tau)\Delta} B_Nv),~&\tau\leq t\leq T.
     \end{array}
     \right.
  \end{eqnarray}
 Several facts are stressed: First, by the conclusion (i) of this lemma, $B_Nv\in H^{-s}(\mathbb{R}^d)$; Second,  in the second line on the right hand side of \eqref{wang973.6},  $\{e^{t\Delta}\}_{t\geq 0}$ is treated as a semigroup on $H^{-s}(\mathbb{R}^d)$; Third, by  the smooth effect of the semigroup, we have that $e^{(\cdot-\tau)\Delta} B_Nv\in C((\tau,T]; L^2(\mathbb{R}^d))$.
  From these facts and \eqref{wang973.6}, we can easily obtain the desired results.

    Hence, we end the proof of Lemma \ref{lemma973.1}.

 \end{proof}

  Now we define, for each $N>0$, a feedback law  $K_N: L^2(\mathbb R^d)\rightarrow l^2(\mathbb Z^d)$ by
 \begin{eqnarray}\label{def-feedback-N}
  K_N g \triangleq  \big\{ \langle g, -f_{N,n} \rangle\big\}_{n\in \mathbb Z^d}
  \;\;\mbox{for each}\;\;
  g\in L^2(\mathbb R^d).
 \end{eqnarray}

 \begin{lemma}\label{lem-feedback-norm}
  For each $N>0$,
  \begin{eqnarray*}
    \| K_N \|_{\mathcal L(L^2(\mathbb R^d),l^2(\mathbb Z^d))}=N^{-d/2}.
  \end{eqnarray*}
 \end{lemma}

 \begin{proof}
  Arbitrarily fix $N>0$.
  On one hand, it follows by (\ref{def-feedback-N}) and (ii) of Lemma \ref{lem-orth} in Appendix that for each $g\in L^2(\mathbb R^d)$,
  \begin{eqnarray*}
    \| K_N g \|_{l^2(\mathbb Z^d)}^2
    =  \sum_{n\in \mathbb Z^d}  \langle   g,   f_{N,n} \rangle^2
    =   N^{-d} \sum_{n\in \mathbb Z^d}   \langle   g,  N^{d/2} f_{N,n} \rangle^2
    \leq   N^{-d} \|g\|^2,
  \end{eqnarray*}
 which yields that
  \begin{eqnarray*}
  \| K_N \|_{\mathcal L(L^2(\mathbb R^d),l^2(\mathbb Z^d))}  \leq  N^{-d/2}.
  \end{eqnarray*}
  On the other hand, by making use of (\ref{def-feedback-N}) and (ii) of Lemma \ref{lem-orth} in Appendix again,
  we find  that
  \begin{eqnarray*}
   \| K_N \|_{\mathcal L(L^2(\mathbb R^d),l^2(\mathbb Z^d))}
   \geq
   \|K_N (N^{\frac{d}{2}}f_{N,0})\|=
   N^{-\frac{d}{2}}\|N^{\frac{d}{2}}f_{N,0}\|_{l^2(\mathbb{Z}^d)}^2=N^{-\frac{d}{2}}.
  \end{eqnarray*}
  Thus, the desired result follows at once. This ends the proof of Lemma \ref{lem-feedback-norm}.

 \end{proof}

 \begin{remark}
  Lemma \ref{lem-feedback-norm} implies that the bigger the density of the lattice points $\big\{ \frac{n}{N} \big\}_{n\in\mathbb Z^d}$ becomes, the smaller the norm of the feedback law $K_N$ is.

 \end{remark}

The main result of this section is the next theorem.
 \begin{theorem}\label{theorem4.6}
 There is $C_1(d)$ so that for each  $\varepsilon\in(0,1)$ and each $y_0\in L^2(\mathbb{R}^d)$,
 \begin{eqnarray*}
   \left\| y\left(T; N,y_0,K_N y(t; N,y_0,0)|_{t=\tau-}\right) \right\|\leq  \varepsilon \|y_0\|,
   \end{eqnarray*}
 when $N\geq C_1(d) \sqrt{\frac{1}{T-\tau} (1+\ln\frac{1}{\varepsilon}) }$.
  \end{theorem}

 To prove Theorem \ref{theorem4.6}, we need the next lemma.

 \begin{lemma}\label{thm-eq-ob-control-feedback}
  Let  $N>0$ and $\varepsilon>0$. Then the following two statements are equivalent:

  (i) For each $u_0\in L^2(\mathbb R^d)$,  there exists $R_{u_0} \in L^2(\mathbb R^d)$ with $\| R_{u_0} \| \leq   \varepsilon \|u_0\|$ so that
  \begin{eqnarray*}
u(T,\cdot) = \sum_{n\in \mathbb{Z}^d} u\big(T-\tau, {n}/{N}\big)
e^{\tau \Delta } f_{N,n}(\cdot) + R_{u_0}(\cdot)\;\;\mbox{in}\;\;L^2(\mathbb{R}^d),
\end{eqnarray*}
where $u$ solves Equation (\ref{equ-heat}) with  $u(0,x)=u_0(x)$, $x\in \mathbb{R}^d$.

  (ii) For each $y_0 \in L^2(\mathbb R^d)$,
  \begin{eqnarray*}
   \| y(T; N,y_0,v_{y_0}) \|\leq\varepsilon \|y_0\|,
  \end{eqnarray*}
  where $v_{y_0}:=K_N y(t; N,y_0,0)|_{t=\tau-}$.

 \end{lemma}

 \begin{proof}
 Arbitrarily fix $N>0$ and $\varepsilon>0$.
 We organize the proof by the following two steps:

 \vskip 5pt
\noindent \textit{ Step 1. We show that (i)$\Longrightarrow$(ii).}

 Suppose that (i) holds. Arbitrarily fix $y_0\in L^2(\mathbb R^d)$. Set
 \begin{eqnarray}\label{def-v-y0-s1-1}
    v_{y_0} := K_N e^{\tau\Delta}y_0.
 \end{eqnarray}
 It is clear that
 \begin{eqnarray*}
   y(T; N,y_0,v_{y_0}) =  e^{T\Delta} y_0  + e^{(T-\tau)\Delta} B_N v_{y_0}.
 \end{eqnarray*}
 From this, (\ref{def-v-y0-s1-1}), (\ref{def-control-operator-N}) and (\ref{def-feedback-N}), it follows that for each $u_0\in L^2(\mathbb R^d)$,
 \begin{eqnarray*}
  & & \big\langle y(T; N,y_0,v_{y_0}),u_0 \big\rangle
  \nonumber\\
   &=&\langle y_0, e^{T\Delta} u_0 \rangle
   + \Big\langle B_N v_{y_0}, e^{(T-\tau)\Delta} u_0 \Big\rangle_{H^{-d}(\mathbb R^n),H^d(\mathbb R^n)}
   \nonumber\\
   &=&\langle y_0, u(T,\cdot) \rangle
   +  \sum_{n\in\mathbb Z^d}  u (T-\tau,n/N)\langle e^{\tau\Delta}y_0,-f_{N,n} \rangle
   \nonumber\\
   &=&\Big\langle y_0, u(T,\cdot)-\sum_{n\in \mathbb Z^d} u(T-\tau,n/N) e^{\tau\Delta} f_{N,n}(\cdot) \Big\rangle,
 \end{eqnarray*}
 where $u$ solves Equation (\ref{equ-heat}) with  $u(0,x)=u_0(x)$, $x\in \mathbb{R}^d$. The above, together with the statement (i), yields that
 \begin{eqnarray*}
   \|y(T; N,y_0,v_{y_0})\|
   = \sup_{\|u_0\|\leq 1}   \big\langle y(T; N,y_0,v_{y_0}), u_0 \big\rangle
   \leq  \sup_{\|u_0\|\leq 1}  \langle y_0, R_{u_0} \rangle
   \leq \varepsilon \|y_0\|,
 \end{eqnarray*}
 where $R_{u_0}$ is given by the statement (i).
 Hence,  the statement (ii) is true.

  \vskip 5pt
  \noindent \textit{Step 2. We prove  that (ii)$\Longrightarrow$(i).}

   Assume that (ii) is true. Arbitrarily fix $u_0\in L^2(\mathbb R^d)$. Write $u$ for the solution to Equation (\ref{equ-heat}) with  with  $u(0,x)=u_0(x)$, $x\in \mathbb{R}^d$. Then we get from (\ref{def-control-operator-N}) and (\ref{def-feedback-N}) that for each $y_0\in L^2(\mathbb R^d)$,
  \begin{eqnarray*}
  & &  \Big\langle y_0, u(T,\cdot)-\sum_{n\in \mathbb Z^d} u(T-\tau,n/N) e^{\tau\Delta} f_{N,n}(\cdot) \Big\rangle
  \nonumber\\
  &=& \langle y_0,e^{T\Delta} u_0 \rangle
   -  \sum_{n\in\mathbb Z^d} (e^{(T-\tau)\Delta} u_0)(n/N)\langle e^{\tau\Delta}y_0,f_{N,n} \rangle
   \nonumber\\
   &=& \langle e^{T\Delta}y_0, u_0 \rangle
  + \Big\langle B_N K_N e^{\tau\Delta}y_0, e^{(T-\tau)\Delta} u_0 \Big\rangle
  \nonumber\\
  &=& \big\langle y(T; N,y_0,\hat v_{y_0}), u_0  \big\rangle,
  \end{eqnarray*}
  where $\hat v_{y_0}:= K_N e^{\tau \Delta} y_0$.
  This, along with the statement  (ii), yields that
  \begin{eqnarray*}
  \|R_{u_0}\| &:= &  \| u(T,\cdot)-\sum_{n\in \mathbb Z^d} u(T-\tau,n/N) e^{\tau\Delta} f_{N,n}(\cdot)\|
   \nonumber\\
   &=& \sup_{\|y_0\|\leq 1}  \Big\langle y_0, u(T,\cdot)-\sum_{n\in \mathbb Z^d} u(T-\tau,n/N) e^{\tau\Delta} f_{N,n}(\cdot) \Big\rangle
   \nonumber\\
   &=& \sup_{\|y_0\|\leq 1}   \big\langle y(T; N,y_0,\hat v_{y_0}), u_0(\cdot) \big\rangle
   \leq \varepsilon \|u_0\|,
  \end{eqnarray*}
  which leads to the statement (i).

  Hence, we end the proof of Lemma \ref{thm-eq-ob-control-feedback}.
 \end{proof}

 \begin{proof}[Proof of Theorem \ref{theorem4.6}]

  Let $C(d)$ be the constant  given by (i) of Theorem \ref{thm-decomp}.
  Arbitrarily fix  $\varepsilon\in(0,1)$.
 It is clear that there is  $C_1(d)$ so that
  \begin{eqnarray}\label{wang4.92022}
  C_1(d)\sqrt{\frac{1}{T-\tau} \Big(1+\ln\frac{1}{\varepsilon}\Big) }\geq \sqrt{\frac{1}{T-\tau} \Big(\ln 3 + |\ln (2C(d))|+\ln\frac{1}{\varepsilon} \Big) }.
  \end{eqnarray}
  Arbitrarily take $N$ so that
   \begin{eqnarray}\label{equ-920-1}
      N\geq    C_1(d)\sqrt{\frac{1}{T-\tau} \Big(1+\ln\frac{1}{\varepsilon}\Big) }.
  \end{eqnarray}
By \eqref{wang4.92022} and \eqref{equ-920-1}, we can easily check that
  \begin{eqnarray}\label{last-pf-def-aug-3}
       C(d)\left(1+(\sqrt{T-\tau}N)^{-\frac{d}{2}}\right)e^{-(T-\tau)N^2} \leq 2C(d)e^{-(T-\tau)N^2}\leq \varepsilon.
  \end{eqnarray}
Then by (i) of Theorem \ref{thm-decomp} and (\ref{last-pf-def-aug-3}), we
  see that the statement (i) of Lemma \ref{thm-eq-ob-control-feedback}
  is true for the  above $(N,\varepsilon)$. Thus,  by Lemma \ref{thm-eq-ob-control-feedback}, we get the statement (ii) of Lemma \ref{thm-eq-ob-control-feedback},
  which is exactly the desired result. This ends the proof of Theorem \ref{theorem4.6}.
 \end{proof}

 \begin{remark}\label{wangremark3.3}
(i) Theorem \ref{theorem4.6} gives a special kind of feedback null approximate controllability
for the system (\ref{heat-impulsive-control-system-N}). It
 says that  for each  $\varepsilon\in(0,1)$ and $y_0\in L^2(\mathbb{R}^d)$,
 by taking $N\geq C_1(d) \sqrt{\frac{1}{T-\tau} (1+\ln\frac{1}{\varepsilon}) }$,
 the system (\ref{heat-impulsive-control-system-N}), with the control operator $B_N$
 and the feedback control $v=K_N y(t; N,y_0,0)|_{t=\tau-}$, drives  $y_0$   at time $0$ into the ball $B_{\varepsilon\|y_0\|}(0)\subset L^2(\mathbb{R}^d)$ at time $T$. (According to Lemma \ref{lemma973.1}, $y(t; N,y_0,0)|_{t=\tau-}$ exists in
 $L^2(\mathbb{R}^d)$.)
 This controllability differs from usual approximate controllability/ null approximate controllability
   from the following perspectives: First, in our case, the  control operator depends on  $\varepsilon$, while
 in
  usual ones  (see, for instance, \cite{AE}, \cite{AEWZ},
  \cite{FZ1}, \cite{FZ}, \cite{LZZ},
  \cite{PWX}  and \cite{QW}),  control operators are independent of $\varepsilon$;
  Second, our controls are active only on some lattice points in the space, while controls in usual cases are active on open
  or measurable subsets of spaces;
  Third, in our case, control has a feedback form, while in
  usual ones, controls are open-looped.

(ii)
The following statement  is not true: There is  $N>0$ so that
\begin{eqnarray*}
   \| y(T; N,y_0,K_N y(t; N,y_0,0)|_{t=\tau-}) \|_{L^2(\mathbb R^d)}  \leq  \varepsilon \|y_0\|
   \;\;\mbox{for each}\;\;
   \varepsilon>0
   \;\;\mbox{and}\;\;
   y_0\in L^2(\mathbb{R}^d).
  \end{eqnarray*}
 This can be easily proved by Lemma \ref{thm-eq-ob-control-feedback} and the conclusion (iii) of Theorem \ref{thm-decomp}. We omit the details.
  \end{remark}

At the end of this section, we give an application of  Theorem \ref{thm-asym-ob-balls} to some feedback controllability. Let $N>0$ and  $r>0$.
Define $\mathcal{H}_{N,r}:=\{ n\in \mathbb{Z}^d\;:\; |n/N|<r\}$ and write $M_{N,r}$
for the number of all elements in $\mathcal{H}_{N,r}$. Next, we define a control operator $B_{N,r}$ from $\mathbb R^{M_{N,r}}$ to $H^{-d}(\mathbb R^d)$ in the following manner:
 \begin{eqnarray*}\label{def-control-operator-N-r}
  B_{N,r}v := \sum_{n\in \mathcal{H}_{N,r}} v_n \delta(\cdot-n/N)
  \;\;\mbox{for each}\;\;
  v=\{v_n\}_{n\in \mathcal{H}_{N,r}} \in \mathbb R^{M_{N,r}},
 \end{eqnarray*}
 and define a feedback law  $K_{N,r}: L^2(\mathbb R^d)\rightarrow \mathbb R^{M_{N,r}}$ by
 \begin{eqnarray*}\label{def-feedback-N-r}
  K_{N,r} g :=  \Big( \langle g, -f_{N,n} \rangle \Big)_{n\in \mathcal{H}_{N,r}}
  \;\;\mbox{for each}\;\;
  g\in L^2(\mathbb R^d).
 \end{eqnarray*}
  We consider the control system (\ref{heat-impulsive-control-system-N}) with $B_N$ being replaced by $B_{N,r}$, and write $y_{N,r}(\cdot,\cdot;y_0,v)$ for its solution. The next Theorem \ref{theorem4.7-added} gives an application
  of Theorem \ref{thm-asym-ob-balls} to some feedback controllability. Its proof
    is quite similar to that of Theorem \ref{theorem4.6}  and will be omitted.

\begin{theorem}\label{theorem4.7-added}
There is $C(d)>0$ so that for each  $\varepsilon\in(0,1)$ and each $y_0\in L^2(\mathbb{R}^d)$,
 \begin{eqnarray*}
  \left( \int_{\mathbb R^d} (1+|x|)^{-2} \Big| y_{N,r}\Big(T,x; y_0,K_{N,r} y_{N,r}(\cdot,t; y_0,0)|_{t=\tau-}\Big)  \Big|^2  \mathrm dx \right)^{\frac{1}{2}}
   \leq  \varepsilon \|y_0\|,
   \end{eqnarray*}
 when $N\geq C(d) \sqrt{\frac{1}{T-\tau} (1+\ln\frac{1}{\varepsilon}) }$ and $r\geq C(d)(1+ T^{\frac{d}{2}}) (1+T^{-\frac{1}{2}}) \varepsilon^{-1}$.
  \end{theorem}

\begin{remark}\label{wang4.81021}
Theorem \ref{theorem4.7-added} says that
for any  $\varepsilon\in(0,1)$ and $y_0\in L^2(\mathbb{R}^d)$,
 by taking $N\geq C(d) \sqrt{\frac{1}{T-\tau} (1+\ln\frac{1}{\varepsilon}) }$ and $r\geq C(d) (1+ T^{\frac{d}{2}}) (1+T^{-\frac{1}{2}}) \varepsilon^{-1}$,
 the system (\ref{heat-impulsive-control-system-N}), with the control operator $B_{N,r}$
 and the feedback control $v=K_{N,r} y(t; N,y_0,0)|_{t=\tau-}$, drives  $y_0$   at time $0$ into the closed ball in the weight space:  $\{f
 ~:~  \int_{\mathbb{R}^d}(1+|x|)^{-2}|f(x)|^2 \mathrm dx <+\infty\}$, centered at the origin and of radius $\varepsilon\|y_0\|$, at time $T$.
\end{remark}

\section{Appendix}

The following results   present some properties
on the family
$\{f_{N,n}\}_{n\in \mathbb{Z}^d}$ (given by  (\ref{equ-basis})).

\begin{lemma}\label{lem-orth}
Let $N>0$.
Then the following conclusions are true:

 (i) For each  $n\in \mathbb{Z}^d$,
  \begin{eqnarray}\label{equ-fourier-1}
  \widehat{f_{N,n}}(\xi)=(2\pi)^{-\frac{d}{2}} N^{-d} e^{-i\frac{n}{N}\cdot\xi}
  \chi_{Q_{\pi N}(0)}(\xi)\;\;\mbox{for all}\;\;\xi\in \mathbb{R}^d.
    \end{eqnarray}

  (ii) The family  $\{N^{d/2}f_{N,n}\}_{n\in \mathbb{Z}^d}$ is an orthonormal set  in $L^2(\mathbb{R}^d)$.

      (iii)    For each  $\{a_n\}_{n\in \mathbb{Z}^d}\in l^2(\mathbb{Z}^d)$,
 the series $\sum_{n\in \mathbb{Z}^d}a_n f_{N,n}$
 converges in $L^2(\mathbb{R}^d)$ and
\begin{align}\label{equ-orth-3}
\Big\|\sum_{n\in \mathbb{Z}^d}a_n f_{N,n}\Big\| = N^{-d/2} \Big(
\sum_{n\in \mathbb{Z}^d} |a_n|^2 \Big)^{1/2}.
\end{align}

\end{lemma}

\begin{proof}
(i)   Define
\begin{eqnarray*}
 f(x):= \prod_{j=1}^d\frac{\sin \pi x_j}{\pi x_j}\;\;\mbox{for each}\;\;x=(x_1,\dots,x_d)\in \mathbb{R}^d.
\end{eqnarray*}
By some direct computations, we obtain that
\begin{eqnarray*}
f(x) = (2\pi)^{-d} \int_{\mathbb R^d} e^{-ix\cdot\xi} \chi_{Q_{\pi}(0)}(\xi) \mathrm d\xi \;\;\mbox{for all}\;\;x=(x_1,\dots,x_d)\in \mathbb{R}^d.
\end{eqnarray*}
Then by the inverse Fourier transform, we get that
\begin{eqnarray*}
 \hat f(\xi)= (2\pi)^{-\frac{d}{2}}
  \chi_{Q_{\pi}(0)}(\xi)\;\;\mbox{for all}\;\;\xi\in \mathbb{R}^d.
\end{eqnarray*}
From this, (\ref{equ-basis}) and the properties on the Fourier transform related to the translation and  the scaling,  we obtain (\ref{equ-fourier-1}). So the conclusion (i) is true.

(ii)
By  the Parseval-Plancherel formula  and the conclusion (i) of this lemma, we see that
\begin{eqnarray*}\label{equ-orth-2}
\langle f_{N,n},f_{N,n'} \rangle
&=&\int_{\mathbb{R}^d} \widehat{f_{N,n}}(\xi) \overline{\widehat{f_{N,n'}}(\xi)} \,\mathrm d\xi
\nonumber\\
&=&(2\pi )^{-d}N^{-2d} \int_{\xi\in Q_{\pi N}(0)} e^{-i\frac{(n-n')}{N}\cdot\xi} \,\mathrm d\xi
 \nonumber\\
&=&
\begin{cases}
0, \quad \quad \, \,n\neq n';\\
N^{-d}, \quad n=n'.
\end{cases}
\end{eqnarray*}
From this, the conclusion (ii) follows.

(iii) The conclusion (iii) follows from (ii) of this lemma and Corollary 2 in Section 4 of Chapter III in \cite{Yosida} (see Page 88 in \cite{Yosida}). This  ends the proof of Lemma \ref{lem-orth}.

\end{proof}

We now give the proof of Theorem  \ref{lem-shannon}.

\begin{proof}[Proof of Theorem  \ref{lem-shannon}]
Arbitrarily fix $N>0$ and $f\in \mathscr{P}_{N}$. We claim that
\begin{align}\label{equ-shannon-3}
\widehat{f}(\xi) = \sum_{n\in \mathbb{Z}^d} f\left({n}/{N}\right) (2\pi )^{-\frac{d}{2}} N^{-d} e^{-i\frac{n}{N}\cdot\xi}
  \chi_{Q_{\pi N}(0)}(\xi)\;\;\mbox{for each}\;\;\xi\in\mathbb R^d.
\end{align}
 First,  $\left\{(2\pi N)^{-\frac{d}{2}}e^{i\frac{n\cdot\xi}{N}}\right\}_{n\in\mathbb{Z}^d}$ is a complete
orthonormal system of $L^2\big(Q_{\pi N}(0)\big)$. Second,  the restriction of the function $\widehat{f}$ over $Q_{\pi N}(0)$ is in $L^2\big(Q_{\pi N}(0)\big)$.
By these two facts, we see that
\begin{align}\label{equ-shannon-4}
\widehat{f}(\xi)=\sum_{n\in \mathbb{Z}^d} a_n (2\pi N)^{-\frac{d}{2}}e^{i\frac{n\cdot\xi}{N}}\;\;\mbox{for each}\;\;\xi\in Q_{\pi N}(0),
\end{align}
where
\begin{eqnarray*}
a_n = \int_{Q_{\pi N}(0)} (2\pi N)^{-\frac{d}{2}} e^{-i\frac{n\cdot\xi}{N}}
\widehat{f}(\xi)\,\mathrm d\xi.
\end{eqnarray*}
From the above and  Fourier's inversion theorem, one can easily see that
\begin{align}\label{equ-shannon-5}
a_n = N^{-\frac{d}{2}} f\left(-{n}/{N}\right).
\end{align}
Now, (\ref{equ-shannon-3}) follows from (\ref{equ-shannon-4}), (\ref{equ-shannon-5})
and the fact that $\mbox{supp}\widehat{f}\subset Q_{\pi N}(0)$.

Next, the first equality in \eqref{equ-shannon} follows from  (\ref{equ-shannon-3}),
(i) of Lemma \ref{lem-orth} (i.e., (\ref{equ-fourier-1}))
and Fourier's inversion theorem.

Finally, by the first equality in \eqref{equ-shannon} and (iii) of Lemma \ref{lem-orth}, we get
the second equality in \eqref{equ-shannon}. This ends the proof of Theorem \ref{lem-shannon}.

\end{proof}

 \end{document}